\documentclass{amsart}
\usepackage{amsfonts}
\usepackage{stmaryrd}
\usepackage{amssymb}

\newtheorem{theorem}{Theorem}[section]
\newtheorem{lemma}[theorem]{Lemma}
\newtheorem{corollary}[theorem]{Corollary}

\theoremstyle{definition}

\theoremstyle{remark}
\newtheorem{remark}[theorem]{Remark}

\numberwithin{equation}{section}

\newcommand{\abs}[1]{\lvert#1\rvert}

\begin{document}

\title{\bf On the Ohno-Nakagawa Theorem}

\author{Xia Gao}
\address{School of Mathematics, Peking University, Beijing 10087, China PRC}
\email{xia@math.pku.edu.cn}

\subjclass[2010]{}

\date{November 30, 2015}


\keywords{order, conductor, $L$-function, binary cubic form}

\begin{abstract}
 In this paper we give a new proof of the Ohno-Nakagawa Theorem  using
  the techniques of $L$-series.
 By applying Eisenstein's parametrization of binary cubic forms on the one hand,
  and a class field theory interpretation of  Datskovsky \& Wright's Theorem on the other,
  we  reduce the Ohno-Nakagawa Theorem  to an identity involving
  the $L$-series and the truncated $L$-series  of quadratic orders.
 We prove this identity  by establishing a general relation
  among these two types $L$-series.
\end{abstract}

\maketitle

\section{Introduction}

 Let $L$ denote the lattice of integral binary cubic forms\,:
 \begin{eqnarray*}
   L = \{ \, x(u,v) = x_0 u^3 + x_1 u^2 v + x_2 u v^2 + x_3 v^3 \,| \,
   \, x_i \in {\mathbb Z}\, , \, 0 \leq i \leq 3 \,\}\, .
 \end{eqnarray*}
 Put $ \Gamma = SL(2,{\mathbb Z} )\,$. The group $ \Gamma  $ acts on
 $L$ by the formula
 \begin{eqnarray*}
  ( \gamma x )(u,v) = x( (u,v) \gamma ) \, , \quad  \forall \,
  \,\gamma \in \Gamma , \, x \in L \, .
 \end{eqnarray*}
 Let $ \mbox{disc} (x) $ denote the  discriminant of  $ x(u,v) \,$,
  i.e.\,,
 \begin{eqnarray*}
 \mbox{disc}(x)
  = x_1^2 x_2^2 + 18 x_0 x_1 x_2 x_3
    - 4 x_0 x_2^3 - 4 x_1^3 x_3 - 27 x_0^2 x_3^2  \,.
 \end{eqnarray*}
 It is  invariant under  the action of $\,\Gamma \, $.
 For $x \in L \,$, let $\Gamma_x $ denote the isotropic subgroup of
 $x$ in $\Gamma  $. It is known that $ |\, \Gamma_x | $ is $1$ if
 $\mbox{disc}(x) < 0 \, $ and is $1$ or $3$ if $\mbox{disc}(x) > 0 $.
 Let
 \begin{eqnarray*}
 \quad L^{\vee} = \{ \, x(u,v) =
   x_0 u^3 + x_1 u^2 v + x_2 u v^2 + x_3 v^3 \in L \,| \,
   \, x_1 , x_2 \in 3{\mathbb Z}\,  \,\}\,
 \end{eqnarray*}
  be the dual lattice of $L$ with respect to
  the invariant alternating form\,:
\begin{eqnarray*}
 \quad \quad \langle \, x\, , y \,\rangle \,= \,x_3 y_0 - \frac{1}{3}\,  x_2 y_1 +
 \frac{1}{3}\, x_1 y_2 -   x_0 y_3 \,.
\end{eqnarray*}
 It is clear that $L^{\vee}$ is  an $\Gamma$-invariant lattice.
 Put
\begin{eqnarray*}
 \quad \quad L_{\pm} =\{ \, x \in L \,| \, \pm \mbox{disc} (x) > 0 \,\}\,
   \quad \mbox{ and }
  \quad L^{\vee}_{\pm} =\{ \, x \in L^{\vee} \,| \, \pm \mbox{disc}
   (x) > 0 \,\}\, .
\end{eqnarray*}
 In a seminal paper \cite{Shin}, Shintani introduced the following
  Dirichlet series now bearing his name\,:
\begin{eqnarray*}
 \xi_{1} (s)\, =
  \sum_{ x \in \,{\Gamma \backslash \,L_{+}} }
     \frac{1}{ | \,\Gamma_x |}\,
      |\, \mbox{disc} (x)\, |^{-s} \,\, ,
       \quad  \xi_{2} (s)\, =
  \sum_{ x \in \,{\Gamma \backslash \,L_{-}} }
   |\, \mbox{disc} (x)\,|^{-s} \,\, ,
\end{eqnarray*}
\begin{eqnarray*}
  \quad \xi_{1}^{\vee} (s) \,
  = \sum_{ x \in \,{\Gamma \backslash \,L_{+}^{\vee}} }
   \frac{1}{ |\,\Gamma_x | }\, | \,\mbox{disc} (x) /  27\, |^{-s} \,\, ,
    \quad
  \xi_{2}^{\vee} (s)\, =
   \sum_{ x \in \,{\Gamma \backslash \,L_{-}^{\vee}} }
    | \,\mbox{disc} (x)/27 \,|^{-s} \, \, .
\end{eqnarray*}
 Here we scale  Shintani's original series in the dual lattice case by a
  factor of $ 3^{3s}$.
 Using the theory of prehomogeneous vector space, Shintani  proved that
  all these series can be analytically continued to the whole complex plane
  with simple poles at $1$ and $5/6\,$, and satisfy the matrix functional equation
\begin{eqnarray*}
 \binom{ \,\xi_{1} (1-s)\,}{\, \xi_{2} (1-s)\, } = 2^{-1} 3^{\,3s-2}\pi^{-4s}
 \,
 \Gamma \Big(s \,-\frac{1}{6}\,\Big)\,\Gamma (s)^2 \,\Gamma \Big(s \,+
 \frac{1}{6}\,\Big) \\ \quad\quad \times
\left( \begin{array}{cc} \sin 2 \pi s & \sin \pi s  \\
 3 \sin  \pi s & \sin 2 \pi s \end{array} \right)
 \binom{ \,\xi_{1}^{\vee} (s)\,}{\, \xi_{2}^{\vee} (s)\, }  \, .
 \end{eqnarray*}
 In 1997, Ohno made the surprising discovery that Shintani's four
  zeta functions are essentially two functions \cite{Ohno}\,.
 This was later proved by Nakagawa \cite{Nak2}\,.
\begin{theorem}[Nakagawa]
\begin{eqnarray*}
 (1) \quad \xi_{1}^{\vee} (s) = \xi_{2}  (s) \,\; ;
\qquad\qquad (2) \quad \xi_{2}^{\vee} (s) = 3 \, \xi_{1} (s) \,.
 \end{eqnarray*}
\end{theorem}
\noindent
 Nakagawa's proof uses very sophisticated counting arguments.
 In this paper, we give a more streamlined proof by exploiting
  the rich algebraic structure  lying beneath  the four zeta functions.
 In fact, we show that identities (1) and (2) above can be further
  divided into infinitely many identities involving the $L$-series
  and the truncated $L$-series  of quadratic orders.
 For other works related to these two mysterious identities,
  see \cite {CT, CRT, Dioses, OTS,  Os, TT}\,.

 This paper is organized as follows.
 In section $2\,$, we introduce the  notations and definitions.
 In section $3\,$, we give a self-contained introduction to
  Eisenstein's parametrization of binary cubic forms.
 Using this parametrization, we can express the $\xi_{\,i}^{\vee} (s)$'s directly
  in terms of $L$-series of quadratic orders.
 Note that in \cite{Nak2} Nakagawa has to resort to an analog of
  Datskovsky \& Wright's formula to start the counting.
 In section $4\,$, using the techniques of class field theory,
  we interpret Datskovsky \& Wright's  Theorem  in terms of
  the truncated $L$-series of quadratic orders.
 This  reduces the proof of Theorem $1.1$ to an identity involving
  the $L$-series and the truncated $L$-series  of quadratic orders.
 We  prove this identity in section $5\,$ by establishing a general relation
  between  these two types $L$-series.
 In the Appendix, we derive   a simple relation connecting
  the abelian $L$-series  of a number field $k$
  and the $L$-series of  certain orders of $k^{n+1}$.
 This is inspired by a problem arising in Section $4\,$.

\section{Basic Definitions}
 Throughout this paper, we denote the cardinality of a finite set $X$ by $ | X | \,$.
 We let ${\mathbb Z}$ denote the ring of integers
  and ${\mathbb Z}^{+}$ its subset of positive integers.
 Let ($a,b$) denote the greatest common divisor of integers $ a $ and $b\,$.
 Let $\mbox{sgn}(a)$ denote the sign of a real number $a\,$.
 If $\varphi : A \rightarrow B $ is a map and $H$ is a subset of $B$,
  we write $ \varphi^{-1}(H) $ for the inverse image of $H$.
 If $G$ is an Abelian group, we let $G^{\vee}= \mbox{Hom}(\,G, {\mathbb C}^* )$ denote its dual group.
 We assume that all rings $R$ have the identity $1$ and
  all subrings of $R$ share the same identity element of $R\,$.
 Furthermore, we let $R^*$ denote the group formed by the
  invertible elements of $R\,$.

\subsection*{Orders of $\acute{\mbox{e}}$tale algebras}
 Let $A $ be an $\acute{\mbox{e}}$tale algebra over ${\mathbb Q} \,$.
 This means that $ A \simeq K_{1}  \times \cdots \times K_{m}\,$, where
  $K_{i}$ (~$ 1 \leq i \leq m $~) are finite extensions of ${\mathbb Q} \,$.
 We let $\mbox{Tr}_{A/\,{\mathbb Q}}$ denote the trace map and
  $\mbox{N}_{A/\,{\mathbb Q}}$ the norm map  from $A$ to $\mathbb Q \,$.
 Let $ \mathcal{O}_{A}\, $ denote the unique maximal order of $A \,$, and
  $ \Delta_{A} $ the discriminant of $A \,$.
 Thus $\mathcal{O}_{A}\simeq \mathcal{O}_{K_{1}} \times \cdots \times \mathcal{O}_{K_{m}}$
  and $ \,\Delta_{A} = \Delta_{K_{1}} \cdots \Delta_{K_{m}} \,$,
  where  $ \mathcal{O}_{K_{i}} $ is  the maximal order of $ K_{i}\,$,
  and $\Delta_{K_{i}}$ is the discriminant of $ K_{i}\,$, $ 1 \leq i \leq m\,$.

 If $ \mathfrak a $ and $\mathfrak b$ are ${\mathbb Z}$-submodules of $A\,$,
  we let $\mathfrak a \,\mathfrak b$ denote the ${\mathbb Z}$-submodules of $A\,$
  generated by elements of the form $ \alpha \,\beta$ with
  $ \alpha \in \mathfrak a $ and $ \beta \in \mathfrak b\,$.
 We call a finitely generated ${\mathbb Z}$-submodule $ \mathfrak a $ of $A\,$
  a full ${\mathbb Z}$-module if the  ${\mathbb Z}$-rank  of $\mathfrak a $ is
  equal to the ${\mathbb Q}$-dimension of $ A\,$.
 If ${\mathfrak a }$ is a full ${\mathbb Z}$-module of $A$, we write
\begin{eqnarray*}
   {\mathfrak a}^{\vee} \,=\,
   \{\, \beta \in A \, | \,
   \mbox{Tr}_{A/\,{\mathbb Q}} (\, \alpha \, \beta \, ) \in  \mathbb Z \,,
   \, \forall \,\, \alpha \in {\mathfrak a} \, \}
\end{eqnarray*}
  for its dual full ${\mathbb Z}$-module.

 We  call   a subring $\mathcal {O}$ of $\mathcal {O}_{A}$   an  order of $A\,$
  if $\mathcal {O}$ is a full ${\mathbb Z}$-module.
 Ideals of the ring $\mathcal {O}$ are referred to integral ideals of $\mathcal {O}$.
 We call a full ${\mathbb Z}$-module  $ \mathfrak a $ of $A$ a (fractional) $\mathcal {O}$-ideal
  if $\mathcal {O}\mathfrak a = \mathfrak a\, $.
 We say an $\mathcal {O}$-ideal ${\mathfrak a }$ is
  $\mathcal{O}$-invertible
  if there exists an $\mathcal {O}$-ideal ${\mathfrak b }$ such that
  ${\mathfrak a \,\mathfrak b }= \mathcal {O}$.
 Observe that $\mathcal{O}_{A}$-ideals are always $\mathcal{O}_{A}$-invertible.
 The set of invertible $\mathcal {O}$-ideals forms a group under ideal multiplication.
 We denote this group by $I(\mathcal{O})\,$.
 Moreover, let $ P(\mathcal{O}) $ denote the subgroup of $I(\mathcal{O})\,$
  consisting of principal $\mathcal {O}$-ideals
  $ \alpha \,\mathcal {O}$ with $\alpha \in A^*$.
 The quotient group $I(\mathcal{O}) / P(\mathcal{O})$ is  called the
  ideal class group of $\mathcal {O}$ and is denoted by $Cl(\mathcal{O})\,$.
 If ${\mathfrak a }$ is an invertible $\mathcal{O}$-ideal,  we let $[\,{\mathfrak a }\,]$
  denote its ideal class  in $Cl(\mathcal{O})\,$.

 Let ${\mathfrak c}$ be an integral ideal of an order $\mathcal{O}$.
 We say an integral ideal ${\mathfrak a }$ of $\mathcal{O}$ is prime to
  ${\mathfrak c }$ if ${\mathfrak a } + {\mathfrak c } =\mathcal{O} $.
 We say that an invertible $\mathcal{O}$-ideal ${\mathfrak a }$ is prime to
  ${\mathfrak c }$ if ${\mathfrak a }$ can be written as ${\mathfrak a }_1 /{\mathfrak a }_2 \,$,
  where ${\mathfrak a }_1 $ and $ {\mathfrak a }_2 $ are integral invertible
  ideals of $\mathcal{O}$ prime to ${\mathfrak c }\,$.
 All invertible $\mathcal{O}$-ideals prime to ${\mathfrak c }\,$ form a group which is
  denoted by $I(\mathcal{O},{\mathfrak c })\,$.
 The conductor of an order $\mathcal{O}$ of $A$ is defined  by
\begin{eqnarray*}
   {\mathfrak f} \,=\, \{\, \gamma \in A \, | \, \gamma \,
   \mathcal {O}_{A} \subseteq \mathcal
   {O} \,\} \,.
\end{eqnarray*}
 It is the largest integral  ideal of $\mathcal {O}_{A}$ contained in $\mathcal {O}\,$.
 An integral $\mathcal{O}$-ideal ${\mathfrak a }$ which is prime to
  the conductor ${\mathfrak f }$ of $\mathcal{O}$ is
  $\mathcal{O}$-invertible with  inverse
  ${\mathfrak a}^{-1}=\mathcal{O} + {\mathfrak f }\,( \,{\mathfrak a } \,\mathcal{O}_A)^{-1}  $.

\begin{lemma} $(Dedekind)$
 The map  $\tilde{{\mathfrak a }} \mapsto \tilde{{\mathfrak a }} \cap \mathcal{O} $
  is a bijection between the set of integral
  $\mathcal{O}_A$-ideals prime to the conductor ${\mathfrak f }$ of $\mathcal{O} $
  and the set of integral   $\mathcal{O}$-ideals prime to ${\mathfrak f }\,$.
 The inverse map is given by : ${\mathfrak a } \mapsto {\mathfrak a }\,\mathcal{O}_A \,$.
\end{lemma}

\noindent
 We can extend the map
  $\tilde{{\mathfrak a }} \mapsto \tilde{{\mathfrak a }} \cap \mathcal{O} $
  to a homomorphism
  $\phi : I(\,\mathcal{O}_A,{\mathfrak f}\,) \rightarrow I(\,\mathcal{O},{\mathfrak f }\,)\,$
  as follows.
 Take $ \tilde{{\mathfrak a }} \in I(\,\mathcal{O}_A,{\mathfrak f  }\,)$ and
  write it as $ \tilde{{\mathfrak a }} = \tilde{{\mathfrak a }}_1 / \,\tilde{{\mathfrak a }}_2 \,$,
  where $\tilde{{\mathfrak a }}_1$ and $\tilde{{\mathfrak a }}_2$ are integral ideals of
  $\mathcal{O}_A $ prime to ${\mathfrak f }\,$.
 Put ${\mathfrak a }_i = \tilde{{\mathfrak a }}_i \,\cap \,\mathcal{O} \,$,
  $ i = 1, 2\, $, and define $\phi (\,\tilde{{\mathfrak a }}\,) = {\mathfrak a }_1 /{\mathfrak a }_2\,$.
 It is easy to see that $\phi$ is well defined.
\begin{corollary}
 The map
  $\phi :  I(\,\mathcal{O}_A,{\mathfrak f }\,) \rightarrow I(\,\mathcal{O},{\mathfrak f }\,)\,$
  is an isomorphism.
 Its inverse map is given by
  $\phi^{-1}: \, {\mathfrak a } \mapsto {\mathfrak a}\,\mathcal{O}_A \,$,
  $\, \forall \,\,{\mathfrak a } \in I(\,\mathcal{O},{\mathfrak f  }\,)\,$.
\end{corollary}
\noindent
 Put $P(\,\mathcal{O},{\mathfrak f }\,) = I(\,\mathcal{O},{\mathfrak f }\,)
\cap P(\mathcal{O})\, $.
 Then we have
\begin{lemma}
 $ Cl(\mathcal{O}) \simeq I(\,\mathcal{O},{\mathfrak f }\,) / P(\,\mathcal{O},{\mathfrak f }\,)\,$.
\end{lemma}

Let $\mathcal{O}$ and $\mathcal{O}'$ be orders of $A$ such that
  $ \mathcal{O} \subseteq \mathcal{O}'$.
 We say that an $\mathcal{O}$-ideal ${\mathfrak a}$ lies under
  an $\mathcal{O}'$-ideal ${\mathfrak a}'$ if ${\mathfrak a}\,\mathcal{O}'= {\mathfrak a}'$.
 We let $ U(\,\mathcal{O}',\mathcal{O}\,) $ denote the group formed
  by those invertible $\mathcal{O}$-ideals lying under $\mathcal{O}'$.
 Let
  $\varphi_{\mathcal{O}'\!,\mathcal{O}} : Cl(\mathcal{O}) \rightarrow  Cl(\mathcal{O}') $
  denote the  homomorphism sending the
   $\mathcal{O}\,$-ideal class $ [\, \mathfrak a  \,]$
  to the $\mathcal{O}'$-ideal class $[\,\mathfrak a\,\mathcal{O}'\, ] $,
  for every $\mathfrak a \in I(\mathcal{O})\,$.
 The  homomorphism  $\varphi_{\mathcal{O}'\!, \mathcal{O}}$ is surjective
  and  induces the  exact sequence
\begin{eqnarray*}
 1 \,\longrightarrow \, \mathcal{O}^{\,*} \,\longrightarrow\,
 \mathcal{O}'^{\,*} \,\longrightarrow\,
 U(\,\mathcal{O}',\mathcal{O}\,) \,\longrightarrow\,
 Cl(\mathcal{O})\, \longrightarrow \,Cl(\,\mathcal{O}') \,\longrightarrow \, 1
 \,.
 \end{eqnarray*}
 From this we deduce that
 \begin{eqnarray}
 |\,U(\,\mathcal{O}',\mathcal{O}\,) \,| =
\,[\,\mathcal{O}'^{\,*}\!:\mathcal{O}^{\,*} ]
 \,| \,Cl(\mathcal{O})\,|\, / \,| \,Cl(\,\mathcal{O}')\,|\,.
 \end{eqnarray}

 If ${\mathfrak a }$ is an integral invertible ideal of $\mathcal{O}$, then
  $ [\, \mathcal{O}:{\mathfrak a }\,] = [\,\mathcal{O}_A : {\mathfrak a }\,\mathcal{O}_A\,] $.
 We call this index the norm of ${\mathfrak a }\,$, and denote it by ${\bf N} {\mathfrak a }$.
 It is easy to extend the definition of norm to all invertible ideals of $\mathcal{O}$.

 Let $\mathcal {O}$ be an order of $A$  of conductor $\,{\mathfrak f} \,$
  and let $ \mathfrak{A} $ be an ideal class  of $Cl(\mathcal O)$.
 We define the partial zeta function
\begin{eqnarray}
 \zeta( \,s \, , \mathfrak{A} \,, \mathcal{O}\, )
  = \sum_{  {\mathfrak a }\,\in \,\mathfrak{A}  } \,
  {\, {\bf N} {\mathfrak a }^{ -s} }\, ,
\end{eqnarray}
  where $\mathfrak a$ goes through all integral invertible $\mathcal{O}\,$-ideals
  in the ideal class $ \mathfrak{A} $.
 We also define the truncated partial zeta function as
\begin{eqnarray}
 \zeta^{\,*}( \,s \, , \mathfrak{A} \,,
  \mathcal{O}\, )
  = \sum_{  {\mathfrak b } } \, {\, {\bf N} {\mathfrak b }^{- s} }\, ,
\end{eqnarray}
  where $\mathfrak b$ ranges over all integral $\mathcal{O}$-ideals in
   $\mathfrak{A} $ which are prime to  the conductor $ {\mathfrak f} \,$.
 To each character $\chi$ of the class group $ Cl(\mathcal{O})\,$,
  we associate the $L$-series
\begin{eqnarray*}
 L( \,s \, , \chi \,, \mathcal{O} \,) \; = \;
  \sum_{\mathfrak{A} \in Cl(\mathcal{O}) }
  \chi(\mathfrak{A}) \,\zeta( \,s \, , \mathfrak{A} \,, \mathcal{O}\, ) \; =\;
  \sum_{\mathfrak a  }   \,\chi ( \mathfrak a )
  \, {\bf N} {\mathfrak a }^{ - s} \,
\end{eqnarray*}
  and the truncated $L$-series
\begin{eqnarray*}
 L^*( \,s \, , \chi \,, \mathcal{O}\, )\; = \;
  \sum_{\mathfrak{A} \in Cl(\mathcal{O}) }
  \chi(\mathfrak{A}) \,\zeta^{\,*}( \,s \, , \mathfrak{A} \,, \mathcal{O}\,)\,.
\end{eqnarray*}
 Since an   integral prime-to-${\mathfrak f }\,$ $\mathcal{O}$-ideal $\mathfrak b $
   factors uniquely into prime ideals of $\mathcal{O}$ in exactly the same way as
  $\,{\mathfrak b}\,\mathcal{O}_{A}$  does  in $\mathcal{O}_{A}\,$,
  we have the Euler product expansion
\begin{eqnarray*}
 L^*( \,s \, , \chi \,, \mathcal{O}\, )\,
  = \sum_{ (\, {\mathfrak b} \,,\, {\mathfrak f }\,)\,=\, 1 } \,
   \chi ( \mathfrak b )  \,{\bf N} {\mathfrak b }^{- s}
  \,\,\,
  = \,\,\, \prod_{ {\mathfrak p} \, \nmid \, {\mathfrak f } } \,
   \Big( \, 1 - \frac{ \, \chi ( \mathfrak p ) }
      { \, {\bf N} {\mathfrak p }^{ s} }\, \Big)^{-1}  ,
\end{eqnarray*}
  where the product runs over  prime ideals  $\mathfrak p $
  of $\mathcal{O}$ not dividing  the conductor ${\mathfrak f }\,$.

\subsection*{Characters and conductors of class groups}
 Let $k$ be a number field and  let $ {\mathfrak m } $ be a modulus in $k$,
  that is, $ {\mathfrak m }  $ is a formal product ${\mathfrak m }_0\,{\mathfrak m }_\infty $
  where $ {\mathfrak m }_0 $ is an integral ideal  of $\mathcal{O}_k\,$
  and ${\mathfrak m }_\infty$ is a set of real embeddings of $k\,$.
 Let $\alpha \, \in k \,$.
 We say that
  $\alpha \equiv 1 \,( \,\mbox{mod}^{\,*}{ \mathfrak m }\,)$
  if $ v_{\mathfrak p }(\,\alpha -1\, ) \geq v_{\mathfrak p }(\,{\mathfrak m}_0)$
  for all primes ${\mathfrak p }$ of $\mathcal{O}_k $ dividing ${\mathfrak m }_0\,$,
  and if $\sigma ( \,\alpha ) > 0 $
  for all embeddings $\sigma \in {\mathfrak m }_\infty\,$.
 We let $I_k(\mathfrak m) = I(\mathcal{O}_k\,,{\mathfrak m}_0) $ denote
  the group formed by those $\mathcal{O}_k$-ideals prime to ${\mathfrak m}_0\,$,
  and let $ P_{k,1} (\mathfrak m) $ denote  its  subgroup
  consisting of principal ideals $\alpha \mathcal{O}_k$ with
  $ \alpha \equiv 1 \, (\,\mbox{mod}^{\,*}{\mathfrak m }\,)\,$.
 The quotient group  $  I_k(\mathfrak m) \,/P_{k,1} (\mathfrak m)$
  is  called the ray class group of $k$ modulo $\mathfrak m\,$,
  and is denoted by $ Cl_k (\mathfrak m)$.

 For each pair of moduli ${\mathfrak m }= {\mathfrak m }_0 \,{\mathfrak m}_\infty $
  and ${\mathfrak m }'={\mathfrak m }_0' \,{\mathfrak m }_\infty'$
  with ${\mathfrak m }' |\,{\mathfrak m }\,$,
  (i.e., ${\mathfrak m }_0'\,|\,{\mathfrak m }_0$ and
   ${\mathfrak m }_\infty' \subseteq {\mathfrak m }_\infty$)\,,
  we let $\psi_{k,{\mathfrak m}'\!,{\mathfrak m }} $ denote the surjective homomorphism
\begin{eqnarray*}
 Cl_k ({\mathfrak m}) = I_k({\mathfrak m }) \,/P_{k,1} ({\mathfrak m })
   \, \rightarrow \,
 Cl_k (\,{\mathfrak m}') = I_k(\,{\mathfrak m}') \,/P_{k,1} (\,{\mathfrak m}') \,
\end{eqnarray*}
  induced by the inclusion map
    $I_k({\mathfrak m }) \subseteq I_k(\,{\mathfrak m }')\,$.
 If $\chi'$ is a character of $Cl_k (\,{\mathfrak m }')\,$, then
   $\chi = \chi' \circ \, \psi_{k,{\mathfrak m }'\!,{\mathfrak m }}$
  is a character of $Cl_k ({\mathfrak m })\,$.
 We say in this case that $\chi $ is induced by $\chi'$,
  or $\chi $ is defined at $I_k(\,{\mathfrak m }')\,$.
 If a character $\chi$  of $Cl_k ({\mathfrak m})$ cannot be induced by
  any character $\chi'$ of $Cl_k (\,{\mathfrak m }')\,$
  for any  proper divisor ${\mathfrak m }'$ of ${\mathfrak m}\,$,
  then we say that $\chi$ is  a primitive character of  $Cl_k ({\mathfrak m})$.
 It is a standard fact in class field theory that every character
  $\chi$ of $Cl_k (\mathfrak m)\,$ is induced by a unique primitive
  character of $Cl_k (\,\mathfrak f\,)$ for some modulus $\mathfrak f\, $.
 The modulus $\mathfrak f $ is uniquely determined by  $\chi$
  and is called the conductor of $\chi\,$.

 Similarly, let $H'$ be a subgroup of $Cl_k (\,{\mathfrak m}')$
  and let ${\mathfrak m }$ be a  modulus in $k$ divisible by ${\mathfrak m }'\, $,
  then $ H = \psi_{k,{\mathfrak m }'\!,{\mathfrak m }}^{-1}(\,H')\,$
  is a subgroup of $Cl_k (\mathfrak m)$.
 We say that the subgroup $H $ of $Cl_k (\mathfrak m)$ is induced by
  the subgroup $H'$ of $Cl_k (\,{\mathfrak m}')\,$.
 For each subgroup  $H$ of $Cl_k (\mathfrak m)\,$, there is a unique
  modulus  $\mathfrak f $ and a unique subgroup $H_1$ of $ Cl_k (\,\mathfrak f\,)\,$
  such that $H$ is induced by $H_1$ and
  $H_1$ cannot be further induced by any other subgroup
  of $Cl_k (\,{\mathfrak m}')\,$ for any modulus ${\mathfrak m }'$.
 The modulus  $\mathfrak f $ is called the conductor of $H\,$.

 More generally, let (\,$I$,\,$ \prec $\,) be a partially ordered set.
 We assume for simplicity that for each $j \in I$, the number of indexes
  $ i \in I $ with $ i \prec j $ is finite.
 Let $(\,A_i\,)_{\,i \in I}$ be a family of finite abelian groups and
  suppose we have a family of surjective homomorphisms
   $\varphi_{\,ij}\,: A_j \rightarrow A_i $
   for all $ i\prec j $ with the following properties : \\
  \indent  1) $\varphi_{\,ii}$ is the identity map on $A_i \,$; \\
  \indent  2) $ \varphi_{\,im} = \varphi_{\,ij} \circ \varphi_{j\,m} $
             for all  $ i \prec j \prec m \,$; \\
  \indent  3) If $ i \prec m $ and $ j \prec m \,$,
     then there exists a $l \in I$ with
     $ l \prec i $ and  $ l \prec j $ such that \\
  \indent
     $ \quad \quad \mbox{Ker}\,\varphi_{\,im}\, \,\mbox{Ker}\,\varphi_{j\,m}
        = \mbox{Ker}\,\varphi_{\,lm}\,$. \\
 Then the pair
   $(\,(\,A_i\,)_{\,i \,\in \,I}, (\,\varphi_{\,ij}\,)_{ \,i\prec j \,\in \, I }) $
  is called a compatible inverse system of abelian groups over $I$,
  and the homomorphisms $\varphi_{\,ij}$ are called
  the transition morphism of the system.
 We say that a  subgroup $H$ of $A_j$ is induced by
  a subgroup $H'$ of $A_{i}$ if
   $ i \prec j $ and  $\varphi_{ij}^{-1} (\,H') = H \,$.
 Clearly a subgroup $H$ of $A_j$ is induced from a subgroup of $A_i$
  if and only if $ i \prec j $ and $ \mbox{Ker}\,\varphi_{\,ij} \subseteq H\, $.
 Thus in  a compatible inverse system
  every subgroup $H$ of $A_j$  is induced by a unique subgroup $H_1$ of $A_f$
  for some $ f \prec i $ such that
  $H_1$ cannot be further induced by any other subgroups except itself.
 Such index $f \in I $ is unique and is called the conductor of
  the subgroup $H$.
 One defines the conductors of  characters of $A_j$ in a similar way.
 In fact, the conductor of a character $\chi$ of $A_j$ is
  simply the conductor of the subgroup $ \mbox{Ker}\,\chi$ of $A_j\,$.

 As an example let $k$ be a quadratic  $\acute{\mbox{e}}$tale algebra
  over $ {\mathbb Q} \,$, i.e.\,,
  $k\,$ is either a quadratic field or $ {\mathbb Q} \varoplus {\mathbb Q}\,$.
 Let $\mathcal{O}_{k}$ denote the maximal order of $k$
  and write $\bf 1$ for the identity element of $k\,$.
 For  $m \in {\mathbb Z}^{+}$, we let
  $ \mathcal{O}_{k,m} = {\mathbb Z} \,{\bf 1} + \,m\, \mathcal{O}_{k}\, $
  denote the order of $k$ of conductor $m\,$.
 It is easy to see that for $c\,,\, d \in {\mathbb Z}^{+}$, we have
   $\mathcal{O}_{k,c}\,\mathcal{O}_{k,d} = \mathcal{O}_{k,\,l}\,$,
 where $l = (\,c\,,\,d\,)\,$.
 Write $Cl_{k,m} = Cl(\mathcal{O}_{k,m})$
   for the  class group of $\mathcal{O}_{k,m}\,$.
 Note that in the case $ k ={\mathbb Q}^2$,
  each class of $Cl_{\,{\mathbb Q}^2 \!,\,m}$ contains an  invertible
  $\mathcal{O}_{{\mathbb Q}^2\!,\,m}$-ideal of the form
  $ {\mathfrak a}_{\,t} = {\mathbb Z}\,(\,1\,,t) + {\mathbb
  Z}\,(\,0\,,m)\,$,
  where $t$ is an integer prime to $m\,$,
  and two such ideals $ {\mathfrak a}_{\,t}$ and $ {\mathfrak a}_{s}$ represent the same class
  precisely when $ s \equiv \pm \,t \,\,\mbox{mod}\,m $.
 Since we have $ {\mathfrak a}_{1}\,= \mathcal{O}_{{\mathbb Q}^2,\,m} \,$ and
  $ {\mathfrak a}_{s}\, {\mathfrak a}_{\,t} = {\mathfrak a}_{s\, t }$,
  the map  sending $[\,{\mathfrak a}_{\,t}\,] $ to $ \pm \,t \!\!\mod m
  \,$ is an isomorphism
  $Cl_{\,{\mathbb Q}^2,\,m}\cong ( \, {\mathbb Z} \, / m {\mathbb Z}  \,)^* / \, \{ \pm 1 \}$
   (see also Lemma A$.2$).

 For  $m\,,\,m'\in {\mathbb Z}^{+}$ with $m' |\,m\,$,
  let $\varphi_{k,m'\!,m} : Cl_{k,m} \rightarrow Cl_{k,m'} $
  denote the  surjective homomorphism sending the
   $\mathcal{O}_{k,m}\,$-ideal class $ [\, \mathfrak a  \,]$ to the
   $\mathcal{O}_{k,m'}\,$-ideal class $ [\,\mathfrak a\,\mathcal{O}_{k,m'} \,] $,
  for every  $\mathfrak a \in I(\mathcal{O}_{k,m})\,$.
 The following Lemma shows that for characters
  of the class groups of quadratic orders,
  the concept of conductor is well defined.

\begin{lemma}
 Let $k$ be a quadratic  $\acute{\mbox{e}}$tale algebra.  Then
  $(\,(Cl_{k,m})_{m \,\in {\mathbb Z}^{+}} ,
  (\,\varphi_{k,m'\!,m}\,) $ forms a compatible
   inverse system over
 ${\mathbb Z}^{+}$ partially ordered by the division relation.
\end{lemma}

\begin{proof}
 It suffices to check condition ($3$) of the compatible inverse system.
 We claim that for  $c\,, \,d$
   and $m \in {\mathbb Z}^{+}$ with $c\,|\,m$ and $d\,|\,m\,$,
 \begin{eqnarray*}
  \mbox{Ker}\,\varphi_{k,c,m} \,
 \,\mbox{Ker}\,\varphi_{k,d,m} = \mbox{Ker}\,\varphi_{k,\,l,m} \,,
 \end{eqnarray*}
  where $l = ( \,c, d )$.
 Take any $\mathcal{O}_{k,m}$-ideal class  in $\mbox{Ker}\,\varphi_{k,l,m}$.
 Without loss of generality, we assume that it is represented by an
  invertible $\mathcal{O}_{k,m}$-ideal $\mathfrak a\,$ with
  $ {\mathfrak a}\, \mathcal{O}_{k,\,l}  = \mathcal{O}_{k,\,l}\,$.
 Choose  $u\,,v \in {\mathbb Z}^{+}$ with $c \,|\,  u l $ and $d \,|\,   v l $
  such that $(u\,,v)=1$ and $ u v\,l = m\,$.
 Put
\begin{eqnarray*}
 {\mathfrak a}_1 = u \,{\mathfrak a}\,\mathcal{O}_{k,vl}
  + v \, \mathcal{O}_{k,ul} \quad \mbox{and}
 \quad {\mathfrak a}_2
  = v \,{\mathfrak a}\,\mathcal{O}_{k,ul} + u \, \mathcal{O}_{k,vl} \,.
 \end{eqnarray*}
 It is clear that ${\mathfrak a}_1 $ and ${\mathfrak a}_2 $ are $\mathcal{O}_{k,m}$-ideals.
 They satisfy
\begin{eqnarray*}
 {\mathfrak a}_1 {\mathfrak a}_2 & = & u v \,{\mathfrak a}^2 \,\mathcal{O}_{k,\,l}
  + u^2  \,{\mathfrak a}\, \mathcal{O}_{k,vl}
 +  v^2  \,{\mathfrak a}\, \mathcal{O}_{k,ul} + u v \,\mathcal{O}_{k,\,l}
 \\
 & = & {\mathfrak a}\, ( \,u v  \,\mathcal{O}_{k,\,l}
   + u^2 \,{\mathbb Z}\,{\bf 1} + u^2 \,v l \,\mathcal{O}_{k}
   + v^2 \,{\mathbb Z}\,{\bf 1} + v^2 \,u l \,\mathcal{O}_{k} \,) \\
 & = & {\mathfrak a} \,(\,{\mathbb Z}\,{\bf 1}+ m\,\mathcal{O}_{k} \,) =
 {\mathfrak a}\,.
 \end{eqnarray*}
 This implies that $ {\mathfrak a}_1 $ and $ {\mathfrak a}_2 $ are invertible
  $\mathcal{O}_{k,m}$-ideals.
 Moreover, we have
\begin{eqnarray*}
 {\mathfrak a}_1 \mathcal{O}_{k,c}
  & = &  u \,{\mathfrak a} \,\mathcal{O}_{k,\,l} + v \,\mathcal{O}_{k,c} \\
  & = & u \,{\mathbb Z}\,{\bf 1} + u l  \mathcal{O}_{k}
    + v \,{\mathbb Z}\,{\bf 1} + v c \, \mathcal{O}_{k} \\
  & = & \mathcal{O}_{k,c}
 \end{eqnarray*}
 and similarly $ {\mathfrak a}_2 \mathcal{O}_{k,d} =\mathcal{O}_{k,d} \,$.
 Hence
 $[\,{\mathfrak a}\,] = [\,{\mathfrak a}_1\,]
  \,[ \,{\mathfrak a}_2\,] \in \mbox{Ker}\,\varphi_{k,c,m}\,
  \,\mbox{Ker}\,\varphi_{k,d,m}\, $.
 The other side  of the inclusion is trivial. \end{proof}

\section{Interpreting $\xi_{i}^{\vee} (s) $  in terms of  $L$-series
of Quadratic  Orders}
 Throughout the section, we let $k$ denote a quadratic  $\acute{\mbox{e}}$tale algebra,
  i.e., $k$ is either a quadratic field or $ {\mathbb Q} \,\varoplus \,{\mathbb Q}\,$.
 We define the quadratic order $\mathcal{O}_{k,m}$ and
  its class group $Cl_{k,m}$ as before.
 For $\alpha \in k\,$, we write
  $\alpha^{\tau}$ for the image of $\alpha $ under the unique
  non-trivial automorphism ${\tau}$ of $k\,$.
 When $k$ is a quadratic field, we let $\sqrt{\Delta_k} $ denote
  the positive square root of $\Delta_{k}$ if $\Delta_{k} > 0$ and
  the positive imaginary  square root if $\Delta_{k} < 0$.
 In the case $ k = {\mathbb Q}\, \varoplus \,{\mathbb Q}\,$, we use
  $\sqrt{\Delta_k} $ to denote the element $(-1,1)  \in k\,$, and
  identify $\mathbb Q$ as a subalgebra of $ k$ via the diagonal map
  $ a \mapsto (a\,,a)\,,\,\, \forall \,\,a \in \mathbb Q\,$.

\subsection*{Integral binary cubic forms with middle coefficients
 divisible by $3$ }
 In this section let $n$ be a non-zero integer.
 We want to count the number of $\Gamma$-orbits in
\begin{eqnarray*}
 L^{\vee}(n) = \{ \, x \in L^{\vee} \,| \,  \,\mbox{disc} (x) = - 27 \,n \,\}\,.
\end{eqnarray*}
 As was shown by Eisenstein \cite{Eisen} in the 19th century,
  binary cubic forms of this type
  are parameterized  by  ideal classes of
  their  quadratic resolvent rings.
 Eisenstein proved this  for certain maximal orders.
 Nakagawa   made   this  correspondence explicit in  \cite[Section 2\,]{Nak2}
  for orders of imaginary quadratic fields.
 Bhargava constructed  in his doctoral thesis \cite[Thm 13\,]{Baga}
  a bijective correspondence for general quadratic rings.
 He associated to each $\Gamma$-orbit in $L^{\vee}$
  an equivalent class of the triples $( S, I, \delta )$, where $S$ is an oriented
  quadratic ring, $I$ is an ideal of $S$, and $\delta$ is an
   element of $S \otimes {\mathbb Q}$
  satisfying certain compatibility conditions.
 Since the ideal $I$ is in most cases not $S$-invertible,
  one has  to re-group these triples to get to conditions on ideal class groups.
 In the following  we  give  a self-contained account of the
  parametrization, with proof  based essentially  on one
  equation---the refined syzygy identity $(3.4)$.
 We also state the
  bijection in an explicit and easy to use form.

 Put
\begin{eqnarray}
 \qquad k =\left\{
 \begin{array}{ll}
   {\mathbb Q}( \sqrt{\,n}\, ) \, , & \;
     \mbox{if $ n $ is  not the square of an integer ;} \\
   \, {\mathbb Q} \varoplus {\mathbb Q} \, , &  \;
       \mbox{otherwise\,.}
  \end{array}
  \right.
\end{eqnarray}
 For integers $ b\,,c \geq 1 \,$, set
\begin{eqnarray*}
  \; S_k(\,b\,,c\,) = \{  ( \,{ \mathfrak a } \,, \beta \, )
   \,\, | \,\,
  {\mathfrak a} \, \, \mbox{is an invertible } \mathcal{O}_{k,c}\mbox{\,-ideal}, \,
  \beta \in {\mathfrak a }^3 \mbox{ and }
  {\bf N}(\,\beta {\mathfrak a }^{-3}) = b \, \}.
\end{eqnarray*}
 We call two pairs  $( {\mathfrak a }_1 , \beta_1 )$ and $( {\mathfrak a }_{\,2} , \beta_2 )$
  of $S_k(\,b\,,c\,)$ equivalent, and write
  $( {\mathfrak a }_1 , \beta_1 ) \sim ( {\mathfrak a }_{\,2} , \beta_2 ) $,
  if there exists an element $\rho \in k^{\,*}$ such that
   $ \rho {\,\mathfrak a }_1 = {\,\mathfrak a }_{\,2} $ and $ \rho^3 \beta_1 = \beta_2 \,$.
 We write $ [ \, { \mathfrak a } \,, \beta \,] $ for the equivalence class
  of $( {\mathfrak a } , \beta )\,$ and put
\begin{eqnarray*}
 S (n) = \bigcup_{ ( \,b \,c \,)^2
 \,= \, n / \Delta_k  } S_k(\,b\,,c\,)\,/ \sim \quad (disjoint).
\end{eqnarray*}

 We now define a map $\Psi_n $ from $ S (n) $ to
   $\Gamma \,\backslash \,L^{\vee}(n) $ as follows.
 To each $( { \mathfrak a }\,, \beta ) \in  S_k(\,b\,,c\,) $
  with $( b \,c )^2 \Delta_k =  n$, choose a $ {\mathbb Z} $-basis
  $ \alpha_1 , \,\alpha_2 $ of ${\mathfrak a }$ with
\begin{eqnarray*}
 {\mbox{N}_{k /\,{\mathbb Q}} (\,\beta \,) } \,
  (\,\alpha_1^{\,}\,\alpha_2^{\,\tau} -\alpha_2^{\,}\,\alpha_1^{\,\tau}  )
   \,/ \sqrt{ \Delta_k } < 0 \,
\end{eqnarray*}
   and  put
\begin{eqnarray*}
  x(u,v)
   & = & \frac{1}{ \,c \sqrt{\Delta_k } \,\,{\bf N} { \mathfrak a }^{\,3} }\,
    \big( \,{ \beta }\,
      (\, \alpha_1^{\,\tau} u + \alpha_2^{\,\tau} v \,)^3
       -{\, \beta^{\,\tau} } ( \,\alpha_1 u + \alpha_2 v\, )^3 \,
    \big) \, \\
    & = & \mbox{Tr}_{k/{\mathbb Q}}\,
   \big( \,
    \frac{\beta}{ \,c \sqrt{\Delta_k }\,\,{\bf N} { \mathfrak a }^{\,3} }\,
     ( \,\alpha_1^{\,\tau} u + \alpha_2^{\,\tau} v \,)^3 \,
   \big) \, .
\end{eqnarray*}
 We claim that $ x(u,v) \in L^{\vee}(n)$ and the $\Gamma$-orbit of $ x(u,v) $
  depends only on  $ [ \, { \mathfrak a } \,, \beta \,] \,$.
 The condition $\beta \in {\mathfrak a }^{\,3} $ is equivalent to
   $\,\beta \,(\,{\mathfrak a }^{\tau})^{\,3}
   \subseteq  \,
     {\bf N} { \mathfrak a }^{\,3}  \,\mathcal{O}_{k,c} \,
   = c \,\sqrt{\Delta_k }\,
     {\bf N} { \mathfrak a }^{\,3}  \,\mathcal{O}_{k,c}^{\,\vee} \,$,
  or $ x(u,v) \in L^{\vee}$.
 Write $ x(u,v) = w \big( (u,v)\,g \big) $ with
\begin{eqnarray}
  w(u,v)= \frac{1}{c \sqrt{\Delta_k }\,
     \,{\bf N} { \mathfrak a }^{\,3}  \,}
     \,( \,{ \beta } \,u^3 - \beta^{\,\tau}  v^3 \,)\,
     \quad \mbox{and} \quad
   g = \left[
    \begin{array}{cc}
      \alpha_1^{\,\tau} & \alpha_1  \\
      \alpha_2^{\,\tau} & \alpha_2
    \end{array} \right],
\end{eqnarray}
 we see that
   $ \mbox{disc}(x) =( \mbox{det} g )^6 \,\mbox{disc}(w)= -27\,( b\,c)^2 \Delta_k \,$.
 Moreover, if $( { \mathfrak a }_1 , \beta_1 ) $ is another representative for
   $ [ \, { \mathfrak a } \,, \beta \,] \,$, say
   $ {\mathfrak a }_1 = \rho \,{ \mathfrak a }$ and $ \beta_1 = \rho^{\,3} \beta $
   for some $ \rho \in k^{\,*}$, and
   $ \omega_1 ,\, \omega_2 $ is a $ {\mathbb Z} $-basis of ${\mathfrak a }_1$ with
   $ \,(\, \omega_1 \,\omega_2^{\,\tau} - \omega_2 \,\omega_1^{\,\tau} ) \,
    {\mbox{N}_{k /{\mathbb Q}}  (\,\beta_1) }\, / \sqrt{ \Delta_k} < 0 \, $.
 Then there exists a $\gamma \in \Gamma$ such that
   $ (\,\omega_1 , \omega_2 ) = (\, \rho\, \alpha_1 ,\rho \,\alpha_2)\,\gamma\, $,
  and so
\begin{eqnarray*}
  \mbox{Tr}_{k/{\mathbb Q}}\,
    \big( \,
     \frac{\beta_1}{c  \sqrt{\Delta_k} \,\,{\bf N} { \mathfrak a }_1^{\,3}}\,
      ( \,\omega_1^{\,\tau} u \,+\, \omega_2^{\,\tau} v \,)^3 \,
    \big)
   =  \pm \,x\big( (u,v) \gamma^{T} \big)
\end{eqnarray*}
  is  in the  $\Gamma$-orbit of $x(u,v)\,$.
 We let  $\Psi_n  $ be the map   sending the equivalence class
  $  [ \, {\mathfrak a } \,, \beta \,] $ to the $\Gamma$-orbit of $ x(u,v)\, $.
\begin{theorem}
 The map $\Psi_n : S (n) \rightarrow \Gamma \,\backslash \,L^{\vee}(n) $ is a bijection.
\end{theorem}

\begin{proof}
 We   prove the  subjectivity of   $\Psi_n$  by using a refined syzygy identity.
 Let $k$ denote the quadratic $\acute{\mbox{e}}$tale algebra defined by (3.1) and
  let $d $ be the positive integer satisfying  $ d^{\,2} \Delta_k = n \,$.
 Take any
  $ x(u,v) = x_0 u^3 + 3 \,x_1 u^2 v + 3 \,x_2 u v^2 + x_3 v^3 \in L^{\vee}(n) \,$.
 Let
\begin{eqnarray}
  H(u,v)  =  - \frac{1}{36} \left| \begin{array}{cc}
  \frac{ \partial ^2 x }{ \partial u ^2 } & \frac{ \partial ^2 x }{ \partial u \partial v }
  \\[5pt]
  \frac{ \partial ^2 x }{ \partial v \partial u } & \frac{ \partial ^2 x }{ \partial v ^2 }
 \end{array} \right|  \quad
 \mbox{and} \quad  J(u,v)  =  \frac{1}{3} \left| \begin{array}{cc}
  \frac{ \partial  x }{ \partial u  } & \frac{ \partial x }{  \partial v }
  \\[5pt]
  \frac{ \partial H }{ \partial u } & \frac{ \partial H }{ \partial v  }
 \end{array} \right|
 \end{eqnarray}
 denote respectively the Hessian  and the Jacobian covariant form of $x(u,v)\,$.
 Thus
   $ H(u,v)  = B_0 u ^2  + B_1 u v + B_2 v ^2\,$, where
\begin{eqnarray*}
  B_0  =  x_1^2 - x_0 x_2 \, , \quad  B_1  =  x_1 x_2 - x_0 x_3 \, ,
  \quad B_2  =  x_2^2 - x_1 x_3 \, ,
\end{eqnarray*}
  and $ J(u,v)  = C_0 u ^3  + 3\, C_1 u^2 v + 3\, C_2 u v^2 + C_3 v ^3\, $, where
\begin{eqnarray*}
  C_0  =  x_0 B_1 - 2 \,x_1 B_0 \, , \quad  C_1  =  - x_1 B_1 + 2 \,x_0 B_2 \,, \\
   C_2  =  x_2 B_1 - 2 \,x_3 B_0 \, , \quad C_3  =  - x_3 B_1 + 2 \,x_2 B_2 \,.
 \end{eqnarray*}
 Apply suitable $\Gamma$-action on $x(u,v)$ if necessary, we may assume that  $B_0 \neq 0 \,$.
 These covariant forms satisfy the syzygy relation
\begin{eqnarray*}
   J^2 - \mbox{disc} (H) \,x^2 = 4 H^3
 \end{eqnarray*}
  with $\mbox{disc} (H) = -\, \mbox{disc} (x)/27 = n \,$.
 What we need is the following  refined identity, which can be checked straightforwardly\,:
\begin{eqnarray}
   J + d  \sqrt{\Delta_k }\, x = (\, C_0 + d  \sqrt{\Delta_k }\, x_0
   ) \,(\, u + \frac{ B_1 - d  \sqrt{\Delta_k }\,}{2 B_0}\, v \,)^3\,.
\end{eqnarray}
 Following \cite{Nak2} we let $ b =  \gcd( B_0 , B_1 , B_2 ) $
 denote the content of $H(u,v)$ and put
\begin{eqnarray*}
  h(u,v)  = \frac{1}{b}\, H(u,v)  = b_0 u ^2  + b_1 u v + b_2 v ^2 \,.
\end{eqnarray*}
 Then $\mbox{disc}(h) = c^{\,2} \Delta_k $ for some positive integer $c\,$ with $b\,c=d\,$.
 Put
\begin{eqnarray*}
 \alpha = \frac{ b_1 + c  \sqrt{\Delta_k }\,}{2 \, b_0}
  \quad \mbox{and} \quad
 \beta  =  \frac{  C_0 + d  \sqrt{\Delta_k } \,x_0 \,}{2 \,b \, b_0^{\,3}\, } \,.
\end{eqnarray*}
 It is easy to check, for example using \cite [Thm 5.12]{Gao}\,,
  that ${\mathfrak a }= {\mathbb Z}\, {\bf 1} + {\mathbb Z}\, \alpha \,\,$
  is an invertible fractional ideal of
  $ \,\mathcal{O}_{k,c} = {\mathbb Z } \,{\bf 1} + {\mathbb Z } \,b_{\,0} \,\alpha \,$
  with
  ${\mathfrak a }^{-1} = {\mathbb Z}\,b_{\,0}  {\bf 1}
   + {\mathbb Z}\,( b_{\,0}\, \alpha + b_{\,1} {\bf 1} ) \,$.
 Moreover, we have ${\bf N} { \mathfrak a } = \abs{\,b_{\,0}\,}^{-1}$ and
  $\,\mbox{N}_{k /{\mathbb Q}} (\,\beta \,)   = b \,/\,
  b_{\,0}^{\,3}\,$.
 Thus  $\,{\,\bf N} (\,\beta \, { \mathfrak a }^{-3}) = b \,$ and
   $\,( \,\alpha^{\tau} - \alpha \,)\,
    \mbox{N}_{k /{\mathbb Q}} (\,\beta \,) \, /\sqrt{\Delta_k } < 0 \,$.
 Now $(3.4)$ can be written as
\begin{eqnarray*}
  J + d  \sqrt{\Delta_k }\, x
  = 2\, b \, b_{\,0}^{\,3}\, \beta \, (\,u + \alpha^{\tau} v )^3 \,.
\end{eqnarray*}
 So we have
\begin{eqnarray*}
x(u,v)  =  \mbox{Tr}_{k/{\mathbb Q}}\, \big( \,
  \frac{\beta}{ \,c \sqrt{\Delta_k }\,\,{\bf N} { \mathfrak a }^{\,3} }\,
  ( \, u + \alpha^{\tau} v \,)^3 \, \big) \,.
\end{eqnarray*}
 The condition $ x(u,v) \in L^{\vee}$ implies that  $\beta \in {\mathfrak a }^{\,3} $.
 Hence $   ( { \mathfrak a } \,, \beta ) \in S_k(\,b\,,c\,) $ and
  the map $\Psi_n$  sends $ [\, { \mathfrak a } \,, \beta \,] $
  to the $\Gamma$-orbit of $ x(u,v) \,$.

 The injectivity of $\Psi_n$ follows from  the covariance property of $H(u,v)$ and $J(u,v)\,$.
 Suppose $\Psi_n$ maps the equivalence class of $( { \mathfrak a } \,, \beta )\in S_k(\,b\,,c\,) $
  to the $\Gamma$-orbit of $ x(u,v) \in L^{\vee}(n)\,$.
 Without loss of generality, we assume  ${\mathfrak a }$ has a $ {\mathbb Z} $-basis
  $ \alpha_1 , \,\alpha_2 $  with
  $\, ( \,\alpha_1 \,\alpha_2^{\tau} - \alpha_2 \,\alpha_1^{\tau}\, )\,
   { \mbox{N}_{k /{\mathbb Q}} (\,\beta \,)}\, / \sqrt{ \Delta_k } < 0  $ such that
\begin{eqnarray*}
  \mbox{Tr}_{k/{\mathbb Q}}\,
    \big( \,
  \frac{\beta}{ \,c \sqrt{\Delta_k }\,\,{\bf N} { \mathfrak a }^{\,3} }\,
  ( \,\alpha_1^{\tau} u + \alpha_2^{\tau} v \,)^3 \,
    \big) =   x(u,v) \,.
 \end{eqnarray*}
 Let $w(u,v)$ and $g$ be defined by $(3.2)$.
 Since the Hessian form of $w(u,v)\,$  is
 \begin{eqnarray*}
   H_w(u,v) =\,\mbox{sgn} ( \,\mbox{N}_{k /{\mathbb Q}} (\,\beta \,) \,)
   \frac{b^3}{\,n\,{\,\bf N} {\mathfrak a }^3 } \, u \,v\,,
\end{eqnarray*}
 the Hessian form of $x(u,v) = w \big( (u,v)\,g \big)\,$ is given by
\begin{eqnarray*}
 \qquad H_x(u,v) = \,( \mbox{det} g )^2 \,  H_w \big( (u,v)\,g \big)
   =\,\mbox{sgn} ( \,\mbox{N}_{k /{\mathbb Q}} (\,\beta \,) \,)
   \frac{b}{{\,\bf N} {\mathfrak a } } \,{\bf N} ( \,\alpha_1 u + \alpha_2 v\,) \,.
\end{eqnarray*}
 Hence $b$ is the content of $H_x(u,v)\,$.
 Consequently, $b$ and $c$ are uniquely determined by the $\Gamma$-orbit of $x(u,v)\,$.
 Similarly, using  the  covariance property of the Jacobian  form $J_x(u,v)$ of $x(u,v)\,$, we get
\begin{eqnarray*}
 \frac{1}{\,b}\, J_x(u,v) +  c \sqrt{\Delta_k } \,x(u,v)
  = \frac{2}{\,{\bf N} { \mathfrak a }^{\,3}}
  \, \beta\,( \,\alpha_1^{\tau} u + \alpha_2^{\tau} v \,)^3 .
\end{eqnarray*}
 From this we deduce that $   [\, { \mathfrak a } \,, \beta \,] $ is
  uniquely determined by the $\Gamma$-orbit of $ x(u,v)\,$.

\end{proof}

\subsection*{Transition to $L$-series of quadratic orders}
 For each pair of positive integers $b$ and $c\,$, we put
\begin{eqnarray*}
 \Omega_{k}(\,b\,,c\,)
  = \{ \,\, \mbox{integral invertible $ \mathcal{O}_{k,c}$-ideal}\,
  \, {\mathfrak b}
  \,\, | \,\,
  {\bf N}{ \mathfrak b }= b\,,\, [ \,{ \mathfrak b }\, ] \in
 Cl_{k,c}^{\,3}
 \,\} \,.
\end{eqnarray*}
 Take any $(\, {\mathfrak a } \,, \beta \, )\in S_k(\,b\,,c\,) \,$.
 Clearly $\,\beta \, { \mathfrak a }^{-3} $ is an integral ideal of
  $\mathcal{O}_{k,c} $ contained in $ \Omega_{k}(\,b\,,c\,)\,$
  which depends only on the equivalence class of $( \,{\mathfrak a } \,, \beta \,)\,$.
 We let $\Phi$ denote the map
  from $S_k(\,b\,,c\,) /\sim \,$ to $\,\Omega_{k}(\,b\,,c\,)$
  sending  $[\, {\mathfrak a } \,,\beta\,]\,$ to $\,\beta \,{ \mathfrak a }^{-3}$.
 Furthermore, put
\begin{eqnarray*}
 Cl_{k,c}^{\,(3)} = \{ \, [\,{\mathfrak a }\,] \in Cl_{k,c}
 \,|\,[\,{\mathfrak a }\,]^{\,3} = 1 \, \} \,.
\end{eqnarray*}

\begin{lemma}
 Under the map $\,\Phi \,$,
  each integral $ \mathcal{O}_{k,c}$-ideal in $\,\Omega_{k}(\,b\,,c\,)$
  has exactly
  $|\,Cl_{k,c}^{\,(3)}\,| \,|\,\mathcal{O}_{k,c}^{\,*}\, / \mathcal{O}_{k,c}^{\,*\,3}\,| $
  inverse images.
\end{lemma}

\begin{proof}
 By definition, for each  $ \mathfrak b \in  \Omega_{k}(\,b\,,c\,)\,$, there  exists an
  invertible $ \mathcal{O}_{k,c} $-ideal $ \mathfrak a $ and an
  element $\beta \in k^{*} $ with $( \,{\mathfrak a } \,, \beta \,)\in S_k(\,b\,,c\,) $
  such that $\mathfrak b = \beta \, { \mathfrak a }^{-3}$.
 Furthermore, if $( \,{\mathfrak a }_1 , \beta_1 \,)\in S_k(\,b\,,c\,) $
  is a pair satisfying $\,\beta_1^{\,} {\mathfrak a }_1^{-3} = {\mathfrak b}\,$,
  then  $ [\,{\mathfrak a }_1 / {\mathfrak a }\,]^{\,3} = 1 \, $.
 So we have $|\,Cl_{k,c}^{\,(3)}\,|$ choices for the ideal class of ${\mathfrak a }_1 \,$.
 Once an ideal class is chosen and a representative ${\mathfrak a }_1 $ is fixed,
  $\,\beta_1 $ is determined by ${\mathfrak b}$ up to a unit in $\mathcal{O}_{k,c}\,$,
  hence the equivalence class  $[\, {\mathfrak a }_1 , \beta_1 \,]$ is
  determined by ${\mathfrak b}$ up to
  $ \,|\, \mathcal{O}_{k,c}^{\,*} / \mathcal{O}_{k,c}^{\,*\,3}\,|  $ possibilities.
\end{proof}

\begin{theorem}
 With the notation above, we have
\begin{eqnarray*}
 \quad \xi_{1}^{\vee} (s) \,
   =  \sum_{ k \;\mbox{\footnotesize{imaginary quadratic}}}
    \frac{1}{\,|\,\Delta_k \,|^{s } }\, \,\,
   \sum_{c \,=\, 1}^{\infty} \frac{ 1 }{\,\, c^{\,2 s}\, }
   \sum_{ \chi \,\in\, Cl_{k,c}^{\vee}\,,\, \chi^3 = 1 }
     L( \,2 s \, , \chi \,, \mathcal{O}_{k,c}\, )\,
\end{eqnarray*}
 \,\, and
\begin{eqnarray*}
\xi_{2}^{\vee} (s) \,\,
    =   \sum_{ k \;\mbox{\footnotesize{real quadratic}} }
    \frac{3}{\,|\,\Delta_k \,|^{s } }\, \,
  \sum_{c \,=\, 1}^{\infty} \frac{ 1 }{\,\, c^{\,2 s}\, }
  \sum_{ \chi \,\in \,Cl_{k,c}^{\vee}\,,\, \chi^3 = 1 }
    L( \,2 s \, , \chi \,, \mathcal{O}_{k,c} \,) \\
   \qquad \qquad \qquad \qquad
 + \,\, \sum_{c \,=\, 1}^{\infty} \frac{ 1 }{\,\, c^{\,2 s}\, }
  \sum_{ \chi \,\in \, Cl_{ {\mathbb Q}^2 , c}^{\vee}\,,\, \chi^3 = 1 }
  L( \,2 s \, , \chi \,, \mathcal{O}_{{{\mathbb Q}^2},c} \,)\,,
\end{eqnarray*}
 where in the inner sums above   $\chi$ goes through
  all characters  of $Cl_{k,c} $  $($and $Cl_{ {\mathbb Q}^2 , c}$$)$ satisfying $\chi ^3 = 1 \,$.
\end{theorem}

\begin{proof}
 Let $ x(u,v) \in L^{\vee} $ and let
  $ [\, {\mathfrak a } \,, \beta \,] \in S_k(\,b\,,c\,)/ \sim $
  be its  equivalence class under the bijection  $\Psi_n$.
 It is easy to check, using (\cite[Prop 2.12\,]{Shin}, \cite {Nak2}), that
  $ |\,\Gamma_x \,| = 3 $ if and only if $ k = {\mathbb Q} ( \sqrt{-3\,} \,)\,$ and $ c = 1\,$.
 Thus
\begin{eqnarray*}
 |\, \mathcal{O}_{k,c}^{\,*} \,/ \mathcal{O}_{k,c}^{\,*\,3}\,| =
  \left\{
  \begin{array}{ll}
  3\,|\,\Gamma_x \,|  \, , & \; \mbox{if $k$ is a real quadratic field\,,} \\[4pt]
   \;|\,\Gamma_x \,|  \, , &   \; \mbox{otherwise\,.}
  \end{array}  \right.
\end{eqnarray*}
 Let $ \Omega_{k}(c) = \bigcup_{ b \geq 1} \Omega_{k}(b\,,c) \,$.
 Putting Theorem $3.1$ and Lemma $3.2$ together, we see that
  each ${\mathfrak b } \in \Omega_{k}(c) $ corresponds to
  $ |\,\mathcal{O}_{k,c}^{\,*}\, / \mathcal{O}_{k,c}^{\,*\,3}\,| \, |\,Cl_{k,c}^{\,(3)}\,| $
   $\Gamma$-orbits in $L^{\vee}$  and all these orbits have the  discriminant
    $ - 27 \Delta_k \,c^2 \,{\bf N} {\mathfrak b }^2 $. Thus
\begin{eqnarray*}
 \xi_{1}^{\vee} (s)  & = &
   \sum_{ x \in \,{\Gamma \backslash \,L_{+}^{\vee}} }
   \frac{1}{ |\,\Gamma_x | }\, | \,\mbox{disc} (x) /  27\, |^{-s} \\
   & = &
  \sum_{ k \;\mbox{\footnotesize{imaginary quadratic}}}
   \frac{1}{\,|\,\Delta_k \,|^{s } }\,
   \sum_{c \,=\, 1}^{\infty} \,\frac{ 1 }{ c^{\,2 s} }
   \sum_{ {\mathfrak b } \in \Omega_{k}(c\,) }
   \frac{| \,Cl_{k,c}^{\,(3)}\, |}{ {\bf N} {\mathfrak b }^{2 s} } \,.
\end{eqnarray*}
 The  inner most sum  can be expressed as a summation of
  $L( \,2 s \,, \chi \,, \mathcal{O}_{k,c}\, )\, $ over  all characters $\chi$ of
  $Cl_{k,c} $ satisfying $\chi ^3 = 1 \,$.
 This proves the first part of Theorem $3.3$. Similarly, we have
\begin{eqnarray*}
 \xi_{2}^{\vee} (s) & = &
 \sum_{ k \;\mbox{\footnotesize{real quadratic}}}
  \frac{3}{\,|\,\Delta_k \,|^{s } }\,
  \sum_{c \,=\, 1}^{\infty} \,\frac{ 1 }{ c^{\,2 s} }
  \sum_{ {\mathfrak b } \in \Omega_{k}(c\,) }
   \frac{| \,Cl_{k,c}^{\,(3)}\, |}{ {\bf N} {\mathfrak b }^{2 s} } \\
   &  & \qquad \qquad \qquad + \,
  \sum_{c \,=\, 1}^{\infty} \,\frac{ 1 }{ c^{\,2 s} }
  \sum_{ {\mathfrak b } \in \Omega_{{\mathbb Q}^2}(c\,) }
   \frac{| \,Cl_{{\mathbb Q}^2,c}^{\,(3)}\, | }{ {\bf N} {\mathfrak b }^{2 s} } \,.
\end{eqnarray*}
 Converting to  $L$-series of orders, we obtain the second part of Theorem $3.3\,$.
\end{proof}

\section{Transforming $\xi_{i} (s) $  into Truncated $L$-series of Orders}

 In this section we  express Datskovsky \& Wright's beautiful Theorem \cite {D-W}
  in terms of truncated $L$-series of quadratic orders.

\begin{theorem}[Datskovsky \& Wright]
 The Shintani zeta functions $\xi_{i} (s)\,(i=1,2)$ can be expressed as
\begin{eqnarray}
\frac{ \xi_{1} (s)\,}{ \,\zeta (4s)\, \zeta (6s-1)\, } & = &
  \, 2 \, I + II + \frac{2}{3} \, III+ \frac{1}{3}\,IV \,,\\
 \frac{ \xi_{2} (s)\,}{ \,\zeta (4s)\, \zeta (6s-1)\, } & = &
  \, 2 \, I^{\,\prime} + II^{\,\prime} \, ,
\end{eqnarray}
 where
\begin{eqnarray*}
I & = & \sum_{
  K \,\mbox{\footnotesize{non-Galois cubic field with}} \, \Delta_K > 0   }
  |\Delta_K |^{-s } \,\frac{ \zeta_K (2s) }{ \zeta_K (4s)} \, , \\
 II  & =  &  \sum_{ k \,\,\mbox{\footnotesize{real quadratic field}} }
  |\Delta_k |^{-s } \, \frac{ \zeta (2s) }{ \zeta (4s)} \, \frac{ \zeta_k (2s) }{ \zeta_k (4s)}\,,\\
 III & = & \sum_{ K \,\mbox{\footnotesize{cyclic cubic field} }   }
 |\Delta_K |^{-s }\, \frac{ \zeta_K (2s) }{ \zeta_K (4s)} \, , \\
 IV & =  & \Big(  \frac{ \zeta (2s) }{ \zeta (4s)} \Big)^3 \,
\end{eqnarray*}
 and
\begin{eqnarray*}
 I^{\,\prime} & = & \sum_{ K \,\mbox{\footnotesize{complex cubic field} } }
  |\Delta_K |^{-s } \,\frac{ \zeta_K (2s) }{ \zeta_K (4s)} \, , \\
 II^{\,\prime} & =  &  \sum_{ k \,\mbox{\footnotesize{imaginary quadratic field} } }
  |\Delta_k |^{-s } \, \frac{ \zeta (2s) }{ \zeta(4s)}\, \frac{ \zeta_k (2s) }{ \zeta_k (4s)}  \,  .
\end{eqnarray*}
 In the expressions above conjugate cubic fields are counted only once.
\end{theorem}
\begin{remark}
 Let $ E(3)$ denote the set of isomorphism classes of $\acute{\mbox{e}}$tale algebras
  of dimension $3$ over ${\mathbb Q} \,$.
 Let $\mbox{Aut}(A)$  denote the automorphic group of an $\acute{\mbox{e}}$tale algebra $A \,$
  and $\zeta_A (s)$ the zeta function of $A\,$.
 Then Datskovsky \& Wright's Theorem can be succinctly expressed as
\begin{eqnarray*}
 \xi_{i} (s)  \, = \,\zeta (4s)\, \zeta (6s-1) \!\!\!\!
 \sum_{ \substack{ A \in E(3) \\ (-1)^{i-1}\Delta_A > 0 } }
 \frac{2}{\,| \mbox{Aut}(A) |\,} \, |\Delta_A |^{-s } \,
  \frac{ \zeta_A (2s) }{ \zeta_A (4s)} \,,  \quad  i = 1,2.
\end{eqnarray*}
\end{remark}

 We shall compute these contributions case by case.
 Our  method differs from that of Nakagawa's in that we bring the
  whole Artin $L$-function down to quadratic orders.
 To do this we need to make his map in Lemma $1.10$ of
 \cite{Nak2} (i.e.,  the map $\phi\,'$ in Lemma $4.3$ below) explicit.
 The following Galois eigenspace argument leading to Lemma $4.3$
 are taken directly from Nakagawa's work \cite{Nak2}, with some
  notational changes and  clarifications, especially on the construction of
  the map $\phi\,'$.

\subsection*{Descending to quadratic orders: $K$ is a non-Galois cubic field}
 In this case the Galois closure $\widetilde{K}$ of $K$ contains a unique quadratic field $k\,$,
  and $\widetilde{K}/k\,$ is a cyclic cubic extension whose conductor is $f \, $,
  the positive integer satisfying $ \Delta_K = f^{\,2} \Delta_k \,$.

 Let  $ A_k(f) $ denote the odd part of
  the ray class group $ Cl_k(f) =  I_k(f) \,/P_{k,1} (f)\,$.
 By class field theory, $\widetilde{K}/k$ corresponds to a subgroup $H$ of  $  A_k(f) $
  of index $3\,$ whose conductor is  $f\,$.
 Let $\tau$ be the non-trivial automorphism of the quadratic field $k\,$.
 Then $\tau$ acts on $A_k(f)$ and we have the  decomposition
\begin{eqnarray*}
 A_k(f) = \frac{1+\tau}{2}A_k(f)\oplus \frac{1-\tau}{2}A_k(f) = A_k^{+}(f) \oplus A_k^{-}(f)\, ,
\end{eqnarray*}
  where $\,A_k^{+}(f) =\{ \,a \in A_k(f) \,| \,a^\tau = a \, \}\,$ and
  $\,A_k^{-}(f) =\{\, a \in A_k(f) \,| \,a^\tau = a^{-1} \, \}\,$.
 Since $\widetilde{K}/{\mathbb Q}$ is a Galois extension, ${H}$ is
  invariant under the action of $\tau \,$.
 Thus we have decomposition ${H} = {H}^{+} \oplus \, {H}^{-} $ with
  ${H}^{+}\subseteq A_k^{+}(f)$ and ${H}^{-} \subseteq A_k^{-}(f)\,$.
 Furthermore, since $\mbox{Gal}(\widetilde{K}/{\mathbb Q}) \simeq S_3\,$,
  the action of $\tau$ on  $A_k(f)\, /\, {H}$ is non-trivial.
 Hence ${H}^{+} = A_k^{+}(f)$ and $ [\, A_k^{-}(f) : {H}^{-}] =3 $.
 Conversely, given any subgroup ${H}^{-}$  of $A_k^{-}(f)$ of index $3$  and having conductor  $f$,
  it corresponds to a unique a cyclic cubic extension $\widetilde{K}/k$  such that
  $\widetilde{K}/{\mathbb Q}$ is a Galois extension with
  $ \mbox{Gal}(\widetilde{K}/{\mathbb Q}) \simeq S_3\,$, and
  the three conjugate cubic subfields $K$ of $\widetilde{K}$ have discriminant
  $\Delta_K = f^{\,2} \Delta_k \,$.

 Write $Cl_{k,f} $
   for the  class group of $\mathcal{O}_{k,f}$ as before.
 Let  $\phi : I_k(f) \rightarrow I(\mathcal{O}_{k,f},{f })$ be the
  isomorphism given   in Corollary $2.2\,$.
 Then
\begin{eqnarray*}
  \; P_{k,{\mathbb Z}} ( f ) = \{ \, \alpha \mathcal{O}_k \,\, | \,\,
  \alpha \in k \, ,\, \, \alpha \equiv a \,( \,\mbox{mod}^*{f }\,) \,
  \mbox{ for some } a \in {\mathbb Z} \mbox{ with } (\, a \,,{f }\,) = 1 \, \}\,
\end{eqnarray*}
  is the inverse image of $ P(\,\mathcal{O}_{k,f\,},{ f }\,)\,$ under $\phi\,$.
 By Lemma $2.3$, we have
\begin{eqnarray*}
 Cl_{k,f} \simeq I(\,\mathcal{O}_{k,f\,},{f}\,) /P(\,\mathcal{O}_{k,f\,},{f}\,)
  \simeq  I_k(f) / P_{k,{\mathbb Z}} ( f )\,.
\end{eqnarray*}
 Thus $\phi$ induces a homomorphism $\widetilde{\phi}$
  from the ray class group $ Cl_k(f) $ onto $Cl_{k,f}$
  which extends to the  exact sequence of $\mbox{Gal}(k/{\mathbb Q })$-modules :
\begin{eqnarray}
 1 \,\longrightarrow \, P_{k,{\mathbb Z}} (f)/P_{k,1} (f)
 \,\longrightarrow \, Cl_k(f) \,\stackrel{\widetilde{\phi}}{\longrightarrow} \,
 Cl_{k,f} \, \longrightarrow  \, 1 \, \,.
\end{eqnarray}
 Let $C_k(f) $  denote the odd part of $ Cl_{k,f}\,$.
 Notice that $ a^\tau = a^{-1} $ for all $ a \in C_k(f) $.
 By restricting the exact sequence (4.3) first to the odd part of groups, and then
  to the minus part of the decomposition induced by $\tau$, and using the fact that
  the minus part of $ P_{k,{\mathbb Z}} (f)/P_{k,1} (f)$ vanishes,
  Nakagawa obtained the following  result:
\begin{lemma}
 The restriction  $\phi\,'$ of $\widetilde{\phi} $ to
  the direct summand $A_k^{-}(f)\,$ of $\, Cl_k(f)\,$
  is an isomorphism between $A_k^{-}(f)$ and $ C_k(f)\,$.
\end{lemma}

 Now comes our real innovation of the paper.
 Let $\chi $ be a non-trivial character  of $\mbox{Gal}(\widetilde{K}/k)\,$.
 By induction formula and additivity  of Artin $L$-function, we have
\begin{eqnarray*}
 \zeta_K (s) = \zeta (s) \, L(\,s\,,\chi\,,\widetilde{K}/k\,) \, .
\end{eqnarray*}
 Through  Artin map,   we may regard  $\chi $ as a primitive character
  of the ray class group $ Cl_k(f)$
  and write the Artin $L$-function as
\begin{eqnarray*}
 L(\,s \,,\chi \,,\widetilde{K}/k\,) =
  \prod_{\,\tilde{\mathfrak p }\,\nmid \,f }
   \Big( \, 1-\frac{\chi ( \,\tilde{\mathfrak p} \,)}
    { {\bf N}{\tilde{\mathfrak p} }^s }\,\Big)^{-1} \, ,
\end{eqnarray*}
 where the product is over primes $\tilde{\mathfrak p }$ of ${\mathcal{O}}_k $ not dividing $ f\,$.
 From the  discussion above, we know that  $\chi $ is actually  a character
  of the direct summand $A_k^{-}(f)\,$ of $\, Cl_k(f)\,$, so we may
  further identify $\chi $ through the isomorphism $\phi\,'$ as
  a  cubic character of  $ Cl_{k,f}\,$.
 It is easy to see that $\chi $  is also primitive in the sense of Lemma $2.4$.
 Under these identifications, we have
  $\chi (\,\tilde{{\mathfrak a }} \,) = \chi (\,\tilde{{\mathfrak a }} \cap \mathcal{O}_{k,f} )\,$
  for all integral ideals $\tilde{{\mathfrak a }}$ of $\mathcal{O}_k$ prime to ${f}\,$.
 Thus
\begin{eqnarray}
 L(\,s \,,\chi \,,\widetilde{K}/k\,) = \prod_{\,{\mathfrak p }\,\nmid \,f }
 \Big( \, 1-\frac{\chi ( \,{\mathfrak p} \,)}{ {\bf N}{{\mathfrak p} }^s }\,\Big)^{-1} \,
 = L^*( \,s \, , \chi \,, \mathcal{O}_{k,f} \, ) \,,
\end{eqnarray}
 where in the product  ${\mathfrak p }$ goes over all primes of $\mathcal{O}_{k,f}$ not dividing $ f \,$.

 Since each  pair of complex conjugate primitive cubic characters
  $\chi \, , \, \overline{\chi}\, $ of $ Cl_{k,f}\,$ corresponds bijectively through  $\phi\,'$ with
  an isomorphic class of non-abelian cubic field $K$ whose Galois closure $\widetilde{K}\,$
  is a cyclic cubic extension of $k$ of conductor $ f \, $, we can write
\begin{eqnarray*}
 2 \, I  =   \sum_{ k \;\mbox{\footnotesize{real quadratic}}}
 \frac{1}{\, |\Delta_k |^s } \,\sum_{ f \,= \,1}^{\infty}
 \,\,f^{\,- 2 s} \!\!\!\!\!\!\!\!
 \sum_{
    \substack{ \chi \,\in \,Cl_{k,f}^{\vee}
        \,\mbox{\footnotesize{primitive}} \\
        \chi^3 = \,1 \,, \, \chi \,\neq \,1 }   }
  \,\frac{\,\zeta (2s)\,}{\,\zeta (4s)\,}
   \,\frac{ \,L^*( \,2 s \, , \chi \,, \mathcal{O}_{k,f}\, )}
   {\,\,L^*( \,4 s \, , \chi \,, \mathcal{O}_{k,f}\, )} \, .
\end{eqnarray*}
 If we allow $\chi$ in the inner  sum   take the trivial primitive character of $ Cl_{k,f}$
 (which is  possible only when $f=1$)\,, we need to add the  sum
\begin{eqnarray*}
  \sum_{ k \;\mbox{\footnotesize{real quadratic}}}
 \frac{1}{\, |\Delta_k |^s } \, \ \frac{\zeta (2s)}{\zeta (4s)}
 \,\frac{\,\zeta_k (2s)}{\,\zeta_k (4s)} \, .
\end{eqnarray*}
 But this is exactly the contribution coming from part $II$.
 Thus we have
\begin{eqnarray*}
 2 \, I + II  =  \!\!\! \sum_{ k \;\mbox{\footnotesize{real quadratic}}}
  \!\!\!\!\!\! |\Delta_k |^{- s } \,\sum_{ f \,= \,1}^{\infty}
 \,f^{\,-2 s} \!\!\!\!\!\!
 \sum_{ \substack{ \chi \,\in \,Cl_{k,f}^{\vee} \,\mbox{\footnotesize{primitive}} \\
                     \chi^3 = \,1  }     }
 \frac{\zeta (2s)}{\zeta (4s)} \frac{\,L^*( \,2 s  , \chi ,\mathcal{O}_{k,f}\, )}
 {\,L^*( \,4 s  , \chi , \mathcal{O}_{k,f}\, )}\,.
 \end{eqnarray*}
 Apply the same argument to $I^{\,\prime}$ and $ II^{\,\prime}$, we get
 \begin{eqnarray*}
 2 \, I^{\,\prime} + II^{\,\prime}
 =  \!\!\!\! \!\!\!\! \sum_{ k \,\mbox{\footnotesize{imaginary quadratic}}}
    \!\!\!\! \!\!\! \!\!\!  \!\!\!\!  |\Delta_k |^{- s } \,\sum_{ f \,= \,1}^{\infty}
  \,f^{\,- 2 s}  \!\!\!\! \!\!\! \!
 \sum_{
  \substack{ \chi \,\in \,Cl_{k,f}^{\vee} \,\mbox{\footnotesize{primitive}} \\
   \chi^3 = \,1  }    }
  \frac{\zeta (2s)}{\zeta (4s)} \frac{\,L^*( \,2 s  , \chi , \mathcal{O}_{k,f}\, )}
  {\,L^*( \,4 s  , \chi , \mathcal{O}_{k,f}\, )}\, .
 \end{eqnarray*}

\subsection*{Descending to quadratic orders: $K$ is a cyclic cubic field}
 In this case the conductor of $K/{\mathbb Q} \,$ is a positive integer
  satisfying $ f^{\,2} = \Delta_K \,$.
 By class field theory, cyclic cubic extensions $K$ of $ {\mathbb Q} $ of conductor $f$
  correspond bijectively with pairs of  primitive cubic characters
  $\chi_1^{\,} \, , \, {\chi}_1^{-1}$ of the ray class group
\begin{eqnarray*}
 Cl_{\mathbb Q}(f) =
 I_{\mathbb Q}(f) /P_{{\mathbb Q},1} (f)
 \simeq ( \,{\mathbb Z}\, / {f \,{\mathbb Z} } )^{*} / \{ \pm 1 \}\,.
 \end{eqnarray*}
 Identify $\chi_1^{\,}$ and ${\chi}_1^{-1}$ with
  characters of $\mbox{Gal}(\,K/{\mathbb Q}\,)$ through the Artin map,
  and use the additivity of Artin $L$-functions, we obtain
\begin{eqnarray*}
 \zeta_K(s)  =  \zeta (s) \, L(\,s\,,\chi_1^{\,}\,) \,L(\,s\,,{\chi}_1^{-1})\,,
 \end{eqnarray*}
 where
\begin{eqnarray*}
  L(\,s\,,\chi_1^{\,}\,)= \prod_{\,{ p }\,\nmid \,f }
   \big( \, 1-{\chi_1^{\,} ( \,{ p} \,)} \,  p^{-s} \,\big)^{-1} .
\end{eqnarray*}

 It can be checked directly, or using Corollary A$.5\,$,  that
  there is a bijection between the characters $\chi_1$ of
  $Cl_{\mathbb Q}(f)$ and
  the  characters $\chi$  of  $Cl_{\mathbb Q^2,f}$
  such that $\chi_1$ is primitive if and only if $\chi$ is primitive.
 The bijection is given by
  $\chi (\mathfrak a) = \chi_1^{\,}( a_1/a_0)$, where we put
  $\mathfrak a \,\mathcal{O}_{{\mathbb Q}^2} = ( \,a_0{\mathbb Z}\,, a_1{\mathbb Z}\,)$
  with $a_0, a_1 \in {\mathbb Z}^+ $,
  for all integral  ideals $\mathfrak a $ of
  $\mathcal{O}_{{\mathbb Q}^2,f}$ prime to   $f\,\mathcal{O}_{{\mathbb Q}^2}\,$.
 Furthermore, we have
\begin{eqnarray*}
   L(\,s\,,\chi_1^{\,}\,) \,L(\,s\,,{\chi}_1^{-1}) =
   L^*( \, s  , \chi , \mathcal{O}_{{\mathbb Q}^2,f}\, )\,.
\end{eqnarray*}
 Hence
\begin{eqnarray*}
 2 \, III =  \,
 \sum_{ f \,= \,1}^{\infty} \, \frac{1}{\,f^{\,2 s} } \!\!
 \sum_{ \substack{
    \chi \,\in \, Cl_{\mathbb Q^2,f}^{\vee}\,\mbox{\footnotesize{primitive}} \\
    \chi^3 = \,1 \,, \, \chi \,\neq \,1 } }
  \frac{\zeta (2s)}{\zeta(4s)}\,
  \frac{ L^*( \,2 s  , \chi , \mathcal{O}_{{\mathbb Q}^2,f}\, )}
   {L^*( \, 4 s  , \chi , \mathcal{O}_{{\mathbb Q}^2,f}\, ) } \,.
\end{eqnarray*}
 Again, if we allow $\chi$ to   take the trivial primitive character of $ Cl_{\mathbb Q}(f)$
  in the inner  sum (which is possible only when $f=1$)\,, we  add precisely the  term
\begin{eqnarray*}
 IV = \Big( \frac{\zeta (2s)}{\zeta (4s)} \,\Big)^3   .
\end{eqnarray*}
 Thus
\begin{eqnarray*}
 2 \, III + IV  =   \sum_{\, f \,= \,1}^{\infty} \,
 \frac{1}{\,f^{\,2 s} } \!\!
 \sum_{ \substack{
    \chi \,\in \, Cl_{\mathbb Q^2,f}^{\vee}\,\mbox{\footnotesize{primitive}} \\
    \chi^3 = \,1 \, } }
  \frac{\zeta (2s)}{\zeta(4s)}\,
  \frac{ L^*( \,2 s  , \chi , \mathcal{O}_{{\mathbb Q}^2,f}\, )}
   {L^*( \, 4 s  , \chi , \mathcal{O}_{{\mathbb Q}^2,f}\, ) } \,.
 \end{eqnarray*}
 Substitute  the  formulae above back to Theorem $4.1\,$, we obtain\,:
\begin{theorem}
\begin{eqnarray*}
 \frac{ \xi_{2} (s)\,}{ \,\zeta (2s)\, \zeta (6s-1)\, } \,\,
  = \!\!\! \!\!\! \sum_{ k \;\mbox{\footnotesize{imaginary
  quadratic}}}\!\!\!\! \!\!\!\!\!\!\!\!\!\!
   |\Delta_k |^{- s}  \,\sum_{ f \,= \,1}^{\infty}\, f^{\,- 2 s} \!\!\!\!\!\!\!\!\!\!
  \sum_{ \substack{ \chi \,\in \,Cl_{k,f}^{\vee} \,\mbox{\footnotesize{primitive}} \\
                     \chi^3 = \,1  }    } \!\!
   \frac{\,L^*( \,2 s  , \chi , \mathcal{O}_{k,f}\, )}
    {\,L^*( \,4 s , \chi , \mathcal{O}_{k,f}\, )}\,,
\end{eqnarray*}

\begin{eqnarray*}
 \frac{ \xi_{1} (s)\,}{ \,\zeta (2s)\, \zeta (6s-1)\, } \,\,= \!\!
  \sum_{ k \;\mbox{\footnotesize{real quadratic}}}
    \!\!\!\!\!\! |\Delta_k |^{- s }\,
  \sum_{ f \,= \,1}^{\infty} \,f^{ \,- 2 s} \!\!\!\!\!\!\!\!
   \sum_{ \substack{ \chi \,\in \,Cl_{k,f}^{\vee} \,\mbox{\footnotesize{primitive}} \\
             \chi^3 = \,1  } } \!\!
    \frac{\,L^*( \,2 s  , \chi , \mathcal{O}_{k,f}\,)}
      {\,L^*( \,4 s  , \chi , \mathcal{O}_{k,f}\, )} \\
     \quad \quad + \,\, \frac{1}{3}\,\, \sum_{\, f \,= \,1}^{\infty} \,f^{\,-2 s} \!\!\!\!\!\!
   \sum_{ \substack{
    \chi \,\in \, Cl_{\mathbb Q^2,f}^{\vee}\,\mbox{\footnotesize{primitive}} \\
    \chi^3 = \,1 \, } }
  \frac{ L^*( \,2 s  , \chi , \mathcal{O}_{{\mathbb Q}^2,f}\, )}
   {L^*( \, 4 s  , \chi , \mathcal{O}_{{\mathbb Q}^2,f}\, ) } \,.
\end{eqnarray*}
\end{theorem}

\section{Relations between
  $L^*( \, s \, , \chi \,, \mathcal{O}_{k,m}\, )$
    and $L( \, s \, , \chi \,, \mathcal{O}_{k,m} )$}
 In this section we adopt the following convention.
 Let $k$ be a quadratic  $\acute{\mbox{e}}$tale algebra.
 We  define the quadratic order $\mathcal{O}_{k,m}$
  and its class group $ Cl_{k,m} $ as before.
 Let $\chi$ be a character of $ Cl_{k,m} \,$.
 For economy of  notations, we  denote all characters induced from $\chi$
  by the same symbol $\chi\,$.
 Similarly, if $\chi$ is defined at $ Cl_{k,m'} $
  for some $m'\,|\,m\,$, then we still use $\chi$ to denote
  the character of $ Cl_{k,m'} $ which induces  $\chi\,$.

 The goal of this section is to prove the following identity :
\begin{theorem}
 Let $k$ be a quadratic $\acute{\mbox{e}}$tale algebra
  and let $\chi$ be a primitive character of
  $Cl_{k,f} $ of odd order, then
\begin{eqnarray}
 \sum_{d \,=\, 1}^{\infty} \frac{ 1 }{\,\, d^{\,s}\, }\,
  L( \, s \,,\chi \,, \mathcal{O}_{k ,  f d}\, )
  = \zeta (s)\, \zeta (3s-1)
  \frac{\,L^*( \, s \, , \chi \,,\mathcal{O}_{k,f}\, )}
   {\,L^*( \,2 s \, , \chi \,,\mathcal{O}_{k,f}\, )} \, .
\end{eqnarray}
\end{theorem}
 \noindent Our main result Theorem $1.1\,$  follows easily from this identity.
 By Lemma $2.4\,$,   each character $\chi$ of $Cl_{k,c}$
  is induced by a unique primitive character at $Cl_{k,f}\,$,
   where $f$ is the conductor of $\chi\,$.
 Group  the characters according to their primitive characters,
  and  apply Theorem $5.1$, we get
 \begin{eqnarray*}
  & & \sum_{c \,=\, 1}^{\infty}    \frac{ 1 }{\,\, c^{\,s}\, }
  \sum_{ \substack{ \chi \,\in \,Cl_{k,c}^{\vee}  \\
  \chi^3 = \,1  } } L( \, s \, , \chi \,, \mathcal{O}_{k,c}\, ) \\
   &  = & \sum_{ f \,= \,1}^{\infty} \, \frac{1}{f^{\, s} }
     \sum_{ \substack{ \chi \,\in \,Cl_{k,f}^{\vee}
      \,\mbox{\footnotesize{primitive}} \\
      \chi^3 = \,1  }}  \,\,
      \sum_{d \,=\, 1}^{\infty} \frac{ 1 }{\,\, d^{\,s}\, }\,
        L( \, s \,,\chi \,, \mathcal{O}_{k ,  f d}\, ) \\
   &  = &   \zeta (s)\, \zeta (3s-1) \,\sum_{ f \,= \,1}^{\infty}
   \, \frac{1}{f^{\, s} }
   \sum_{ \substack{ \chi \,\in \,Cl_{k,f}^{\vee}
   \,\mbox{\footnotesize{primitive}} \\
   \chi^3 = \,1  }}
   \,\frac{\,L^*( \, s \, , \chi \,, \mathcal{O}_{k,f}\, )}
      {\,L^*(\,2 s \, , \chi \,, \mathcal{O}_{k,f}\, )} \, .
 \end{eqnarray*}
 Take the sum  over all quadratic $\acute{\mbox{e}}$tale algebras $k$
  with positive (resp., negative) discriminants,
  and apply Theorem $3.3$ and Theorem $4.4\,$, we obtain Theorem $1.1\,$.

\subsection*{From L-series of orders to truncated L-series }
 To compute
  $ L( \, s \,,\chi \,,\mathcal{O}_{k , m})\,$,
  we have to deal with integral $ \mathcal{O}_{k ,  m} $-ideals
  not necessarily prime to the conductor.
 The following Lemma is crucial.
\begin{lemma}
 Let ${\mathfrak a}$ be an integral ideal of $\,\mathcal{O}_{k,m}$ and
  let $ c $ be the smallest positive integer contained in
  $  \, {\mathfrak a} + m \,\mathcal{O}_{k} \,$.
 Put $ m' = m\, /\,c \,\in \, {\mathbb Z} \,$.
 Then $ \,{\mathfrak a} + m \,\mathcal{O}_{k} =  c \,\mathcal{O}_{k,m'} \,$
  and $ c^{-1}{\mathfrak a}\, \mathcal{O}_{k,m'} $ is
  an integral $\mathcal{O}_{k,m'}$-ideal prime to the conductor  of
  $\,\mathcal{O}_{k,m'}$.
\end{lemma}
\begin{proof}
 By modular law, we have
\begin{eqnarray*}
 {\mathfrak a} + m \,\mathcal{O}_{k} & = &
 ( \, {\mathfrak a} + m \,\mathcal{O}_{k}) \,\cap \,
 ( \, {\mathbb Z}\,{\bf 1} + m \,\mathcal{O}_{k}) \\
 & = &  c \, {\mathbb Z}\,{\bf 1} + m \,\mathcal{O}_{k} \\
 & = &  c \,\mathcal{O}_{k,m'} \,.
\end{eqnarray*}
 Thus
    $ \,\, c^{-1} {\mathfrak a}\,\mathcal{O}_{k,m'} +  m' \, \mathcal{O}_{k} =
   \,\mathcal{O}_{k,m'} \,$.
\end{proof}

 For   $m\,,\, m' \in {\mathbb Z}^{+} $ with $ m' |\,m\,$,
  let $\varphi_{k,m'\!,m} : Cl_{k,m} \rightarrow  Cl_{k,m'} $
  denote the surjective homomorphism sending the
  ideal class of $  \mathfrak a  $ to the ideal class of
  $\mathfrak a\,\mathcal{O}_{k,m'}\, $,
  for every invertible $\mathcal{O}_{k,m}$-ideal $\mathfrak a \,$.

\begin{lemma}
For every ideal class $ \mathfrak{A} $ of $ Cl_{k,m} \,$, we have
\begin{eqnarray}
 \quad \quad
  \zeta( \,s \, , \mathfrak{A} \,, \mathcal{O}_{k,m}\, ) =
  \sum_{\,m'|\,m}\,
  \frac{ \,[\,\,\mathcal{O}_{k, m' }^{\,*}:
  \mathcal{O}_{k,m}^{\,*}\,\, ] \,}{ (\,m\,/m')^{ 2 s } }\,
  \zeta^{\,*}( \,s \, , \varphi_{k,m'\!,m}\,(\, \mathfrak{A}\, ) \,,
  \mathcal{O}_{k,m'}\, )\,.
 \end{eqnarray}
\end{lemma}

\begin{proof}
 Let ${\mathfrak a }$ be an integral invertible $\mathcal{O}_{k,m}$-ideal
  belonging to $\mathfrak{A}\, $.
 Let $c$ and $m' $ be positive integers
 such that
 $ ( \, {\mathfrak a} + m\,\mathcal{O}_{k}) \,\cap  \, {\mathbb Z}\,{\bf 1}
  = c \, {\mathbb Z}\,{\bf 1} \,$ and $m' = m\,/\,c\,$.
 Then Lemma $5.2\,$ implies that
 ${\mathfrak b} = c^{-1} {\mathfrak a}\,\mathcal{O}_{k,m'}
  \in \varphi_{k,m'\!,m\,}( \,\mathfrak{A}\, )$
 is an integral ideal of $\mathcal{O}_{k,m'}$
 prime to $ m' \,\mathcal{O}_{k}\,$.
 Conversely, given  integers $c$ and $m'$ with
  $c \,m' =m\,$,
  and given integral $\mathcal{O}_{k,m'}\,$-ideals
  $ {\mathfrak b} \in \varphi_{k,m'\!,m}\,( \,\mathfrak{A}\, )$
  prime to $ m' \,\mathcal{O}_{k}\,$,
  all the invertible $\mathcal{O}_{k,m}$-ideals
  ${\mathfrak a }$ under $  c\,{\mathfrak b}  $ are integral $\mathcal{O}_{k,m}$-ideals,
  as  $ c\, {\mathfrak b} \subseteq c \,\mathcal{O}_{k,m'} \subseteq \mathcal{O}_{k,m}\,$.
 Moreover, all these $\mathcal{O}_{k,m}$-ideals
 $ {\mathfrak a}  $ satisfy
 \begin{eqnarray*}
  \mathcal{O}_k  ( \,  {\mathfrak a} +  m \,\mathcal{O}_{k} \, )
  = c \, \mathcal{O}_{k} ( \,{\mathfrak b} + m' \,\mathcal{O}_{k} \, )
  = c \, \mathcal{O}_{k} \, ,
\end{eqnarray*}
 hence
  $ ( \, {\mathfrak a} + m \,\mathcal{O}_{k} \, )  \cap {\mathbb Z}\,{\bf 1}
    = c \, {\mathbb Z}\,{\bf 1} \,$ by Lemma $5.2\,$.
 It is easy to see that   under each $\mathcal{O}_{k,m'}$-ideal
  $ m\, {\mathfrak b}$,
  there are precisely
  $[\,\mathcal{O}_{k, m'}^{\,*} \!:\mathcal{O}_{k,m}^{\,*}\, ]$
  invertible $\mathcal{O}_{k,m}\,$-ideals lying in the class $\mathfrak{A} \, $.
  The Lemma now follows by grouping the terms in
  $ \zeta( \,s \, , \mathfrak{A} \,, \mathcal{O}_{k,m}\, ) $
  according to  the smallest
  positive integer contained in
  $ {\mathfrak a} + m \,\mathcal{O}_{k}\, $.
\end{proof}

\begin{theorem}
Let $\chi$ be a character of $Cl_{k,m}$ of conductor $f$, then
\begin{eqnarray*}
 L( \, s \,,\chi \,,\mathcal{O}_{k , m}\,) =
  \sum_{f \,|\,m'|\,m}
 \frac{ \,|\,U(\,\mathcal{O}_{k, m' },\mathcal{O}_{k,m}\,) \,|\,
  }{ (\,m\,/\,m')^{\, 2 s } }\,
  L^{*}( \, s \,,\chi \,,\mathcal{O}_{k , m'})\,.
 \end{eqnarray*}
 In particular, if $\chi$ is a primitive character of $Cl_{k,m}\,$,
 then we have
\begin{eqnarray*}
 L( \, s \,,\chi \,,\mathcal{O}_{k , m}\,) =
  L^{*}( \, s \,,\chi \,,\mathcal{O}_{k , m}\,)\,.
 \end{eqnarray*}
\end{theorem}

\begin{proof}
 By Lemma $5.3\,$, we have
\begin{eqnarray*}
 && \!\!\!\!  \!\!\!\!
  L( \, s \,, \chi  \,, \mathcal{O}_{k , m}\,) \,\,
  =  \sum_{  \mathfrak{A} \in Cl_{k,m} }
  \,\chi (\, \mathfrak{A} \, )\,
  \zeta( \,s \, , \mathfrak{A} \,, \mathcal{O}_{k,m}\, ) \\
 & = &  \sum_{\,m'\,|\,m} \,
  \frac{ \,
   [\,\,\mathcal{O}_{k, m'}^{\,*}:\mathcal{O}_{k,m}^{\,*}\,\, ] \,
   }{ (\,m\,/m')^{ 2 s } }
  \sum_{  \mathfrak{A}' \,\in \, Cl_{k,m'} } \,
 \Big( \!\! \!\! \sum_{
  \substack{ \mathfrak{A} \,\in \,  Cl_{k,m} \\
   \varphi_{k,m'\!,m}\,(\, \mathfrak{A}\, )\, = \, \mathfrak{A}' }  }
  \!\!\!\! \chi \,(\, \mathfrak{A}\, )\,\;    \Big)
  \,\, \zeta^{\,*}( \,s \, ,  \mathfrak{A}' \,, \mathcal{O}_{k,m'} )\,.
\end{eqnarray*}
 Since $  \varphi_{k,m'\!,m}$ is surjective,
  given  $\mathfrak{A}'\in Cl_{k,m'}$, there  exists a
  $ \mathfrak{A}^* \in Cl_{k,m}\,$ such that
  $ \varphi_{k,m'\!,m}(\, \mathfrak{A}^*  )\, = \, \mathfrak{A}'$.
 Thus we have
\begin{eqnarray*}
 \sum_{ \substack{
   \mathfrak{A}\, \in \, Cl_{k,m} \\
   \varphi_{k,m'\!,m}(\, \mathfrak{A}\, )\, = \, \mathfrak{A}' }  }
  \!\! \!\!\!\! \chi \,(\, \mathfrak{A}\, )
   & = &
   \chi \,(\, \mathfrak{A}^* ) \!\!
   \sum_{  \mathfrak{B}\, \in \,
       \mbox{\scriptsize{Ker}} \,\varphi_{k,m'\!,m}  }
  \!\! \chi \,(\, \mathfrak{B}\, )\\
   & = &
 \left\{
   \begin{array}{ll}
     {\chi}(\,\mathfrak{A}^*)
     \,|\,\mbox{Ker}\,
     \varphi_{k,m'\!,m}
    \,|\,,
   &  \mbox{if ${\chi}$ is trivial on
   $\mbox{Ker}\,\varphi_{k,m'\!,m}$};  \\[4pt]
      \qquad \qquad 0\,, & \mbox{ otherwise. }
   \end{array}
\right.
 \end{eqnarray*}
 The sum above is non-zero only if  $\chi $ is defined at
 $Cl_{k,m'}$.
 This happens precisely when $m'$ is a multiple of $f\,$.
 The Theorem  follows by applying ($2.1$).
\end{proof}
\begin{remark}
 Let $ \mu $ denote the M$\ddot{\mbox{o}}$bius function on ${\mathbb
 Z}$ and let $\chi$ be a character of $Cl_{k,m}$ of conductor $f$.
 Then we have
\begin{eqnarray*}
 L^{*}( \, s \,,\chi \,,\mathcal{O}_{k , m}\,) =
 \sum_{f \,|\,m'|\,m} \mu (\, m/m' )\,
\frac{ \,|\,U(\,\mathcal{O}_{k, m' },\mathcal{O}_{k,m}\,) \,|\,
  }{ (\,m\,/m')^{ 2 s } }\, L( \, s \,,\chi \,,\mathcal{O}_{k , m'})\,.
 \end{eqnarray*}
 Theorem $5.4$ and the above formula can be vastly generalized to orders of
  arbitrary $\acute{\mbox{e}}$tale algebras. However, Lemma $2.4$
  no longer holds for  orders of $\acute{\mbox{e}}$tale algebras of high degrees.
 \end{remark}

\subsection*{Proof of  Theorem 5.1}

 We begin with  a technical Lemma.

\begin{lemma} For every positive integer $d$, we have
\begin{eqnarray}
  \sum_{ u \,= \,1 }^{ \infty } \,
  \frac{ \,|\,U(\,\mathcal{O}_{k, d }\,,\mathcal{O}_{k,u d}\,) \,|\,
  }{ u^s }
  = \zeta ( s - 1 )
 \,\prod_{ p  \,\nmid \,d } \Big( \,1 - \Big(\frac{\, \Delta_k }{p} \Big)
 \,\frac{1}{p^s}\, \,\Big) \,,
 \end{eqnarray}
  where in the case $k = {\mathbb Q}^2$, it is understood that
  $\big(\frac{\, \Delta_k }{p} \big) =1 $ for all primes $p\,$.
\end{lemma}

\begin{proof}
 Apply the  class number formula of orders (see \cite[p.95]{Lang2}) to (2.1),
  we get
\begin{eqnarray*}
 |\,U(\,\mathcal{O}_{k,d}\,,\mathcal{O}_{k,c}\,) \,|
 = \frac{c}{d} \,\prod_{\substack{ p\,|\,c \\
 p  \,\nmid \,d } } \Big( \,1 - \Big(\frac{ \,\Delta_k }{p} \Big)
 \,\frac{1}{p}\, \,\Big) \,
 \end{eqnarray*}
  for every multiple $c$ of $d\,$.
 The formula is also valid when $k = {\mathbb Q}^2$ by defining
  $\big(\frac{\, \Delta_k }{p} \big) =1 $ for all  $p\,$.
 Write  $u $ in the form $  u_1 u_2 \,$, where $u_1 , u_2 \in {\mathbb Z}^{+}$
  such that
  \,all prime factors of $ u_1 $ are divisors of $ d \,$  and
   $(\,u_2 \,, d \,) = 1 \,$,
  the left hand side of $(5.3)$  becomes
\begin{eqnarray*}
 && \sum_{u_1 } \frac{1}{ u_1^{s-1} }
    \sum_{ u_2 } \frac{1}{ u_2^{s-1}} \,
    \prod_{p\,|\,u_2 } \Big( \,1 - \Big(\frac{ \,\Delta_k }{\,p} \Big)
       \,\frac{1}{p}\, \,\Big) \\
 & = & \prod_{ p \, | \, d }
    \Big( \,1 - \frac{ 1 }{\,p^{s-1}} \Big)^{-1} \;
    \prod_{ p \, \nmid \, d }
      \bigg( \,1 +  \Big( \,\frac{ 1 }{\,p^{s-1}}
         + \frac{ 1}{\,p^{2(s-1)}} + \cdots \, \Big)
    \Big( \,1 - \Big(\frac{\, \Delta_k }{p} \,\Big)
     \frac{ 1 }{p}\, \Big)\;
       \bigg)
 \end{eqnarray*}
 which is  easily seen equal to   the right hand side of $(5.3)$.
\end{proof}

 We now  return to the proof of Theorem $5.1\,$.
 Let $\chi$ be a character of the class groups of orders of $k$ of conductor $f$.
 Apply Theorem 5.4 and Lemma 5.6 to the left hand side of (5.1), we obtain
\begin{eqnarray*}
& & \sum_{d \,=\, 1}^{\infty} \frac{ 1 }{\,\, d^{\,s}\, }\,
 L( \, s\,,\chi\,,\mathcal{O}_{k ,  f d}\, ) \\
& = & \sum_{d \,=\, 1}^{\infty} \frac{ 1 }{\,\, d^{\,s}\, }
 \sum_{m \,| \, d }
 \frac{ \,|\,U(\,\mathcal{O}_{k, f m }\,,\mathcal{O}_{ k , f d }\,) \,|\,
  }{ ( \, d \, / m  )^{ \,2 s } } \, L^{*}( \, s \,,\chi \,,\mathcal{O}_{k , f m}\,)
  \\
 & = & \sum_{m \,=\, 1}^{\infty} \frac{ 1 }{\,\, m^{\,s}\, }\, \sum_{u \,=\, 1}^{\infty}
  \frac{ \,|\,U(\,\mathcal{O}_{k,  f m  }\,,\mathcal{O}_{ k , u f m }\,) \,|\,
  }{  \, u  ^{\, 3 s } } \, L^{*}( \, s \,,\chi \,,\mathcal{O}_{k , f
  m}\,)\, \\
 & = &  \zeta ( \,3 s - 1 ) \sum_{m \,=\, 1}^{\infty}
    \frac{ 1 }{\,\, m^{\,s}\, }
    \prod_{ p  \,\nmid \, f m }
    \Big( \,1 - \Big(\frac{ \Delta_k }{p} \Big) \,\frac{1}{p^{ \,3 s } }\, \,\Big)
   \, \prod_{ {\mathfrak p} \, | \, p } \, \Big( \, 1
 - \frac{ \, \chi ( \mathfrak p ) }{ \, {\bf N} {\mathfrak p }^{ s} } \,
 \Big)^{-1}   ,
\end{eqnarray*}
  where in the last product  we  let $\mathfrak p $ run over all prime ideals
  of $\mathcal{O}_{ k , f  }$  dividing  $p\,$,
  and in the case $k = {\mathbb Q}^2$, we have
  $\,\big(\frac{\, \Delta_k }{p} \big) =1 $ for all $p\,$.
 By bringing those Euler factors  prime to $f$ out of the summation, we get
\begin{equation}
\begin{split}
& \zeta ( \,3 s - 1 )\, L^{*}( \, s \,,\chi \,,\mathcal{O}_{k , f
  }\,)
   \, \prod_{ q  \,\nmid \, f }
    \Big( \,1 - \Big(\frac{ \Delta_k }{q} \Big) \,\frac{1}{q^{ \,3 s } }\, \,\Big)
   \, \\
  & \quad \qquad \times \sum_{m \,=\, 1}^{\infty}
    \frac{ 1 }{\,\, m^{\,s}\, }
    \prod_{ \substack{ p\,|\,m \\
 p  \,\nmid \,f } }
    \Big( \,1 - \Big(\frac{ \Delta_k }{p} \Big) \,\frac{1}{p^{ \,3 s } }\,
    \,\Big)^{-1}
   \, \prod_{ {\mathfrak p} \, | \, p } \, \Big( \, 1
 - \frac{ \, \chi ( \mathfrak p ) }{ \, {\bf N} {\mathfrak p }^{ s} } \,
 \Big)\,.
 \end{split}
 \end{equation}
 Write $m$ in the form $m_1 m_2\,$, where $m_1 , m_2 \in {\mathbb Z}^{+}$
  such that all prime factors of $ m_1 $ divides  $ f \,$  and
  $(\,m_2 \,, f \,) = 1 \,$,
  the summation part in $(5.4)$ becomes
 \begin{eqnarray*}
&& \sum_{m_1 } \,\, \frac{1}{m_1^s} \;\; \sum_{m_2 } \,\,
\frac{1}{m_2^s} \,\prod_{  p\,|\,m_2  }
    \Big( \,1 - \Big(\frac{\, \Delta_k }{p} \Big) \,\frac{1}{p^{ \,3 s } }\,
    \,\Big)^{-1}
    \prod_{ {\mathfrak p} \, | \, p } \, \Big( \, 1
 - \frac{ \, \chi ( \mathfrak p ) }{ \, {\bf N} {\mathfrak p }^{ s} } \,
 \Big) \\[4pt]
 & = & \prod_{ p \, | \, f } \Big( \,1 - \frac{ 1 }{\,p^{s}}\,
 \Big)^{-1} \;
 \prod_{ p \, \nmid \, f } \bigg( \,\,1 \,+ \, \frac{ 1 }{\,p^{s}-1}
 \,\,\Big( \,1 - \Big(\frac{\, \Delta_k }{p} \Big) \,\frac{1}{p^{ \,3 s } }\,
    \,\Big)^{\!-1}
    \prod_{ {\mathfrak p} \, | \, p } \, \Big( \, 1
 - \frac{ \, \chi ( \mathfrak p ) }{ \, {\bf N} {\mathfrak p }^{ s} } \,
 \Big)\, \bigg) \\[4pt]
 & = & \zeta (  s  ) \,\prod_{ p \, \nmid \, f }
 \bigg( \,\,1\, - \,\frac{ 1 }{\,p^{s}} \,+ \, \frac{ 1 }{\,p^{s}} \,
 \Big( \,1 - \Big(\frac{ \,\Delta_k }{p} \Big) \,\frac{1}{p^{ \,3 s } }\,
    \,\Big)^{-1}
    \prod_{ {\mathfrak p} \, | \, p } \, \Big( \, 1
 - \frac{ \, \chi ( \mathfrak p ) }{ \, {\bf N} {\mathfrak p }^{ s} } \,
 \Big)\, \bigg)\,.\\[4pt]
\end{eqnarray*}
 Substitute  back to $(5.4)$, we obtain
\begin{eqnarray*}
 && \sum_{d \,=\, 1}^{\infty} \frac{ 1 }{\,\, d^{\,s}\, }\,
 L( \, s\,,\chi\,,\mathcal{O}_{k ,  f d}\, ) \,
  =  \,\zeta (  s  ) \,\zeta ( \,3 s - 1 )\,
   L^{*}( \, s \,,\chi \,,\mathcal{O}_{k , f
  }\,) \\
  && \qquad \qquad \qquad \times \,\prod_{ p \, \nmid \, f }
 \bigg( \,\Big( \,1\, - \,\frac{ 1 }{\,p^{s}} \, \Big) \,
 \Big( \,1 - \Big(\frac{ \,\Delta_k }{p} \Big) \,\frac{1}{p^{ \,3 s } }\,
    \,\Big)\,+ \, \frac{ 1 }{\,p^{s}} \,
    \prod_{ {\mathfrak p} \, | \, p } \, \Big( \, 1
 - \frac{ \, \chi ( \mathfrak p ) }{ \, {\bf N} {\mathfrak p }^{ s} } \,
 \Big)\, \bigg)\,.
 \end{eqnarray*}
 The proof of Theorem $5.1$ is now reduced to the following Lemma.

\begin{lemma}
 If $\,\chi$ is a  character of $Cl_{k,f}$ of odd order and  $p$ is a
  rational prime not dividing $f\,$, then we have
\begin{eqnarray*}
  \quad \Big( \,1\, - \,\frac{ 1 }{\,p^{\,s}} \, \Big)\,
  \Big( \,1 - \Big(\frac{ \,\Delta_k }{p} \Big)\,\frac{1}{p^{ \,3 s } }\,\,\Big)
  \, + \, \frac{ 1 }{\,p^{\,s}} \,
    \prod_{ {\mathfrak p} \, | \, p } \,
    \Big( \, 1 - \frac{ \, \chi ( \mathfrak p ) }{ \, {\bf N} {\mathfrak p }^{ s} }\,\Big)\,
  = \,\prod_{ {\mathfrak p} \, | \, p } \,
   \Big( \, 1 - \frac{ \, \chi ( \mathfrak p ) }{ \, {\bf N} {\mathfrak p }^{ \,2s} } \, \Big)\,,
\end{eqnarray*}
 where in the  products  $\mathfrak p $ runs through  prime
  ideals of $ \mathcal{O}_{ k , f }$ dividing $p\,$,
  and in the case $k = {\mathbb Q}^2$, it is understood that
  $\big(\frac{\, \Delta_k }{p} \big) =1 $ for all primes  $p\,$.
\end{lemma}

\begin{proof}
 We  claim  that for every $p \nmid f$, there exists an  $a_p \in {\mathbb C}$ such that
\begin{eqnarray}
 \prod_{ {\mathfrak p} \, | \, p } \,
 \Big( \, 1 - \frac{ \, \chi ( \,\mathfrak p ) }{ \, {\bf N} {\mathfrak p }^{ s} }\, \Big)
  =  1 - \frac{a_p}{p^{\,s}} + \Big(\frac{ \,\Delta_k }{p} \Big) \frac{1}{p^{\,2 s}}\,.
\end{eqnarray}
 In the case $ p \,\mathcal{O}_{k,f} = {\mathfrak p}_1 \, {\mathfrak p}_2 $,
  where ${\mathfrak p}_1 \,, {\mathfrak p}_2 $ are  prime ideals of $\mathcal{O}_{k,f}\,$,
  we have $ \big(\frac{ \,\Delta_k }{p} \big)\,=1\,$,
  $ \chi (\,{\mathfrak p}_1) \,\chi (\,{\mathfrak p}_2) = 1 $
  and so $a_p = \chi (\,{\mathfrak p}_1) + \chi (\,{\mathfrak p}_2) \,$.
 In the case $ p \,\mathcal{O}_{k,f} = {\mathfrak p} $ or $ {\mathfrak p}^2$,
  we have $ \chi (\,{\mathfrak p}) =1 $, so
\begin{eqnarray*}
 \prod_{ {\mathfrak p} \, | \, p } \,
 \Big( \, 1 - \frac{ \, \chi ( \,\mathfrak p ) }{ \, {\bf N} {\mathfrak p }^{ s}}\, \Big)
 = \prod_{ {\mathfrak p} \, | \, p } \, \Big( \, 1 - \frac{ \, 1 }{ \, {\bf N} {\mathfrak p }^{ s}}\, \Big)
 = \Big( \, 1 -  \frac{1 }{ \, p^{ s}} \,\Big)
   \,\Big( \, 1- \Big(\frac{ \,\Delta_k }{p} \Big) \,\frac{ 1 }{ \, p ^{ s} } \Big) \,.
 \end{eqnarray*}
 Thus we may take $ a_p = 1 + (\frac{ \,\Delta_k }{p} )\,$.
 The Lemma  follows easily from identity $(5.5)$.

\end{proof}

\begin{remark}
 It is interesting to observe that
  if $\chi$ is a non-trivial primitive cubic character of $Cl_{k,f}\,$,
  and $p$ is a rational prime not dividing $f\,$,  then
   $a_p$ in ($5.5$)  is given by the formula
\begin{eqnarray*}
 a_p = |\, \{ \, [\,u:v\,] \in {\mathbb P}^1(\,{\mathbb F}_p)
   \,|\, x(u,v) = 0 \,\}\,| - 1 \,.
\end{eqnarray*}
 Here   $K$ denotes  the isomorphic class
  of the cubic fields corresponding to $\chi\,$,
 \begin{eqnarray*}
 x(u,v) = x_0 u^3 + x_1 u^2 v + x_2 u v^2 + x_3 v^3
 \end{eqnarray*}
  is an integral primitive binary cubic form belonging to the incomplete canonical class
  of $K$ (i.e., there exists a generator $\theta$ of $K$ such that
  $ x(\,\theta,1)= 0$ and $ \mbox{disc}(x) = \Delta_K $, see \cite{Gao})\,.
 In fact, the  Hasse-Weil zeta function of the zero dimensional variety
  defined by $x(u,v)=0$ is
  \begin{eqnarray*}
   \zeta ( s ) \,
   L( \, s \,,\chi \,,\mathcal{O}_{k , f}\,) = \zeta_K ( s ) \, .
  \end{eqnarray*}
\end{remark}

\appendix


\section{Abelian $L$-series of number fields and
$L$-series of  orders}

 In   this section we let $k$ denote a number field and let
  $ {\mathfrak m } = {\mathfrak m }_0\,{\mathfrak m }_\infty  $
  be a  modulus in $k\,$.
 We associate to each character $\chi$ of
  $ Cl_k (\mathfrak m) = I_k(\mathfrak m) \,/P_{k,1} (\mathfrak m)\, $
 the $L$-series
\begin{eqnarray*}
 L_{k,\mathfrak m}( \,s \, , \chi \,) \; = \;
  \sum_{ ( {\mathfrak a} , {\mathfrak m }_0 ) = 1 } \,
   \chi ( \mathfrak a ) \,{\bf N} {\mathfrak a }^{- s} \,,
\end{eqnarray*}
  where ${\mathfrak a}$ runs through prime-to-${\mathfrak m }_0\,$
  integral ideals  of $\mathcal{O}_{k}\,$.
 We  let $A$ denote the $\acute{\mbox{e}}$tale algebra $ k^{\,n+1} $
  for  some   positive integer $n \,$.
 We write  $ \mathcal{O}_{A} = \mathcal{O}_{k} ^{\,n+1}\,$ for the maximal order of $A$
  and ${\bf 1} $ for the identity element of $A\,$.
 Observe that  the $\mathcal{O}_{A}$-ideals have the form
  $(\,{\mathfrak a}_0,\cdots ,{\mathfrak a}_n ) \,$,  where the
  ${\mathfrak a}_i$'s are fractional ideals of $\mathcal{O}_{k}\,$.
 We let $ \mathcal{O}_{\! A,{\mathfrak m }}$ denote an  order of $A$ of the following form
\begin{eqnarray*}
 \quad \mathcal {O}_k {\,\bf 1} + {\mathfrak m }_0
 \mathcal{O}_{A} =
 \{\,(\,\alpha_0 \,, \cdots , \alpha_n)\in \mathcal{O}_A \,|\,\,
  \alpha_i \equiv \alpha_0 \,\,\mbox{mod}\,{\mathfrak m}_0 \, ,
  \,\, 1 \leq i \leq n \,\}\,.
\end{eqnarray*}
 Note that  $ \mathcal{O}_{\! A,{\mathfrak m }}$
  is an $\mathcal{O}_k$-module
  and   has the conductor ${\mathfrak m }_0 \mathcal{O}_{A} \,$.
 Moreover, we say that an element
   $\alpha = (\,\alpha_0, \cdots ,\alpha_n) \in A^* $
 is equivariant  modulo ${\mathfrak m }_\infty\,$ if
\begin{eqnarray*}
 \quad \quad \sigma
 (\,\alpha_i \, / \alpha_0 ) > 0\, , \,\,\,\,\forall \,\,\,1 \leq i \leq n\,,\,
 \,\forall \,\, \sigma \in {\mathfrak m }_\infty \,.
\end{eqnarray*}
  Put
\begin{eqnarray*}
 \quad \quad
  P_{{\mathfrak m }_\infty} ( \,\mathcal{O}_{\! A,{\mathfrak m }})
  = \{ \, \alpha \,\mathcal{O}_{\! A,{\mathfrak m }} \, |\,
    \alpha \in A^* \mbox{ is equivariant modulo } {\mathfrak m
    }_\infty \} \,.
\end{eqnarray*}
 We call  the quotient group
   $ I (\,\mathcal{O}_{\! A,{\mathfrak m }} ) /P_{{\mathfrak m }_\infty}
    ( \,\mathcal{O}_{\! A,{\mathfrak m }} ) \,$ the narrow class group of
   $\mathcal{O}_{\! A,{\mathfrak m }}$ modulo ${\mathfrak m}_\infty\,$,
 and denote it by $Cl_{A,\mathfrak m}\,$.
 To each  character $\widetilde{\chi}$ of $Cl_{A,\mathfrak m}\,$,
  we associate the $L$-series
\begin{eqnarray*}
  L( \,s \, , \widetilde{\chi} \,,\mathcal{O}_{\! A,{\mathfrak m}} )
   \; = \;
  \sum_{ \tilde{\mathfrak a } \,\subseteq \mathcal{O}_{\! A,{\mathfrak m }}  } \,
   \widetilde{\chi}( \,\tilde{\mathfrak a }\, ) \, {\bf N} {\tilde{\mathfrak a }}^{ - s}\,,
\end{eqnarray*}
 where $\tilde{\mathfrak a }$ goes through  integral invertible
 ideals of $\mathcal{O}_{\! A,{\mathfrak m}}\,$. We also define
  the truncated $L$-series
  \begin{eqnarray*}
  L^*( \,s \, , \widetilde{\chi} \,,\mathcal{O}_{\! A,{\mathfrak m}} )
   \; = \;
  \sum_{ (\tilde{\mathfrak a } ,
   {\mathfrak m }_0 \mathcal{O}_{\!A} )=1 } \,
   \widetilde{\chi}( \,\tilde{\mathfrak a }\, ) \, {\bf N} {\tilde{\mathfrak a }}^{ - s}\,,
  \end{eqnarray*}
  where $\tilde{\mathfrak a }$ ranges over  integral
  ideals of $\mathcal{O}_{\! A,{\mathfrak m}}$
  prime to the conductor of $\mathcal{O}_{\! A,{\mathfrak m}}\,$.

The goal of this section is to derive a simple relation between
  the abelian $L$-series $L_{k,\mathfrak m}( \,s \, , \chi \,)$ of $k$ and
  the $L$-series $L( \,s \, , \widetilde{\chi} \,,\mathcal{O}_{\! A,{\mathfrak m }} \,)$
  of the order $\mathcal{O}_{\! A,{\mathfrak m }}$.

\subsection*{The relations between $L_{k,\mathfrak m}( \,s \, , \chi \,)$  and the Truncated
 $L$-series of $\mathcal{O}_{\! A,{\mathfrak m}}\,$ }

\begin{lemma}
 Let $\tilde{{\mathfrak a}}$ and $\tilde{{\mathfrak b}}$ be
  integral ideals of  $ \,\mathcal{O}_{\! A,{\mathfrak m }}$
  prime to ${\mathfrak m }_0\mathcal{O}_{A} $.
 Suppose there exists an
  $\gamma = (\,\gamma_0, \cdots , \gamma_n\,) \in A^*$
  equivariant modulo ${\mathfrak m }_\infty$ such that
  $\tilde{{\mathfrak b}} = \gamma \,\tilde{{\mathfrak a}}\,$.
 Then
\begin{eqnarray*}
 \qquad \gamma_i \,/ \,\gamma_0 \equiv 1 \,( \,\mbox{mod}^{\,*}{\mathfrak m
  }\,)\,,
  \quad \forall \,\,\, 1 \leq i \leq n\,.
\end{eqnarray*}
\end{lemma}

\begin{proof}
 Since $\tilde{{\mathfrak a}}\,\mathcal{O}_{A}$ and
  $\gamma\,\tilde{{\mathfrak a}}\,\mathcal{O}_{A}$ are integral ideals of
  $\mathcal{O}_{A}$ prime to ${\mathfrak m }_0 \mathcal{O}_{A} $,
  each component $\gamma_i$ of $\gamma$ is prime to ${\mathfrak m }_0\,$.
 Moreover, from
  $ \tilde{{\mathfrak a}} + {\mathfrak m }_0 \mathcal{O}_{A} =
    \,\mathcal{O}_{\! A,{\mathfrak m }}\,$,
  we see that there exists a
  $\alpha = (\,\alpha_0, \cdots ,\alpha_n ) \in \tilde{\mathfrak a} \,$
  such that
  $\alpha_i \mathcal{O}_{k} + {\mathfrak m }_0 = \mathcal{O}_{k}\,$, $ 0 \leq i \leq n\,$.
 Since $\alpha$ and $\gamma\,\alpha $ both belong to $\mathcal{O}_{\! A,{\mathfrak m}}\,$,
  we have
  $\, \alpha_i \equiv \alpha_0 \mbox{\;mod\;}{\mathfrak m }_0\,$
  and
  $\, \gamma_i \,\alpha_i \equiv \gamma_0\,\alpha_0 \mbox{\;mod\;} {\mathfrak m}_0\,$,
  from which we deduce that
  $\gamma_i \,/ \gamma_0 \equiv 1 \,( \,\mbox{mod}^{\,*}{\mathfrak m}\,)\,,
    \,\forall \,\, 1 \leq i \leq n\,$.
\end{proof}

 We  define a map
  $\Psi : Cl_{A,\mathfrak m} \rightarrow Cl_k (\mathcal{O}_k) \times Cl_k (\mathfrak m)^{n}$
  as follows.
 Let $\widetilde{\mathfrak{A}}$ be an ideal class of $ Cl_{A,\mathfrak m}$.
 Choose in $\widetilde{\mathfrak{A}}$ an integral
   $\,\mathcal{O}_{\! A,{\mathfrak m }}$-ideal $\tilde{\mathfrak a}$
   prime to ${\mathfrak m }_0 \mathcal{O}_{A} $.
 Write
   $\tilde{\mathfrak a} \,\mathcal{O}_{A} =
     (\,{\mathfrak a}_0 , \cdots ,  {\mathfrak a}_n )\, $,
  where  the $ {\mathfrak a}_i $'s are integral ideals of $\mathcal{O}_k$
  prime to ${\mathfrak m}_0\,$.
 Let $ \mathfrak{A}_0 $ denote the ideal class of ${\mathfrak a}_0 $
    in $ Cl_k (\mathcal{O}_k)\,$,
  and  $ \mathfrak{A}_i $ ($ 1 \leq i \leq n $) the ray class
  of $ \mathfrak{a}_i / \mathfrak{a}_0 $ in $Cl_k (\mathfrak m)\,$.
 Then
\begin{eqnarray*}
 \quad
  \Psi (\,\widetilde{\mathfrak{A}}\,) = (\, \mathfrak{A}_0 \,,
  \mathfrak{A}_1 \,, \cdots, \mathfrak{A}_n ) \in   Cl_k
  (\mathcal{O}_k) \times Cl_k (\mathfrak m)^{n} \,.
\end{eqnarray*}
 By Lemma  A$.1$, the map $\Psi$ is well defined.
 It is also clear that $\Psi$ is a surjective homomorphism.
 Now suppose  $\Psi (\,\widetilde{\mathfrak{A}}\,) = 1 \,$, that is,
  $\widetilde{\mathfrak{A}}$ is represented
  by some integral
  $ \,\mathcal{O}_{\! A,{\mathfrak m }}$-ideal $\tilde{\mathfrak a}$ with
  $\tilde{\mathfrak a} \,\mathcal{O}_{A}
   = (\, \gamma_0 \mathcal{O}_k \,, \cdots , \gamma_n \mathcal{O}_k ) \,$,
  where the $ \gamma_i $'s  are elements of
  $\mathcal{O}_k$ prime to $ {\mathfrak m}_0 $ such that
  $ \gamma_i \,/ \gamma_0 \equiv 1 \,( \,\mbox{mod}^{\,*}{\mathfrak m})\,$,
     $0 \leq i \leq n\,$.
  Then $\gamma = ( \,\gamma_0\,, \cdots ,
       \gamma_n ) $
   is  in  $ \mathcal{O}_{\! A,{\mathfrak m}} $ and
   equivariant modulo ${\mathfrak m}_\infty\,$.
 Since $\tilde{\mathfrak a}$ and $\gamma \,\mathcal{O}_{\! A,{\mathfrak m}}$
  are both integral ideals of $\mathcal{O}_{\! A,{\mathfrak m}}$ prime
  to ${\mathfrak m }_0 \mathcal{O}_{A} $ and  under
   $\tilde{\mathfrak a} \,\mathcal{O}_{A}$, Lemma $2.1 $ implies that
   $\tilde{\mathfrak a} = \gamma \,\mathcal{O}_{\! A,{\mathfrak m}}\,$.
  This shows that $\Psi$ is  injective.
\begin{lemma}
 The mapping $\Psi :  Cl_{A,\mathfrak m} \rightarrow Cl_k
(\mathcal{O}_k) \times Cl_k (\mathfrak m)^{n}$ is an isomorphism.
\end{lemma}

 Next let $\widetilde{\mathfrak{A}} \in  Cl_{A,\mathfrak m}$ and
 suppose
  $\Psi (\widetilde{\mathfrak{A}}) = (\, \mathfrak{A}_0 ,  \mathfrak{A}_1 ,
   \cdots, \mathfrak{A}_n ) $,
  where $\mathfrak{A}_0 \in   Cl_k (\mathcal{O}_k) $ and
  $\mathfrak{A}_i \in  Cl_k (\mathfrak m) \,$, $1 \leq i \leq n \,$.
 As before, we associate to $\widetilde{\mathfrak{A}}$
   the truncated partial zeta function
\begin{eqnarray*}
  \zeta^{\,*}( \,s \, , \widetilde{\mathfrak{A}} \,,
  \mathcal{O}_{\! A,{\mathfrak m}}\, )
  = \sum_{
   \substack{  \tilde{\mathfrak a }\,\in \,\widetilde{\mathfrak{A}}  \\
   ( \tilde{\mathfrak a } , {\mathfrak m }_0 \mathcal{O}_{\!A} )
     =  1 } } \,
  { {\, {\bf N} \tilde{\mathfrak a }^{ - s} } }\, ,
\end{eqnarray*}
  where $\tilde{\mathfrak a }$ ranges over those  integral
  $\mathcal{O}_{\! A,{\mathfrak m}}$-ideals
  prime to ${\mathfrak m }_0 \mathcal{O}_{\!A} $ lying in the
  class $\widetilde{\mathfrak{A}}\,$.
 Moreover, for each ray class $ \mathfrak A \in Cl_k (\mathfrak m)\,$,  put
 \begin{eqnarray*}
 \zeta_{k,\mathfrak m}( \,s \, , \mathfrak A \,)
  \; = \; \sum_{ \substack{  {\mathfrak a }\,\in \,\mathfrak{A}  \\
   ( {\mathfrak a } , {\mathfrak m }_0 )=  1 } }
    \, \,{\bf N} {\mathfrak a }^{ -s}  \,.
\end{eqnarray*}

\begin{lemma}
 Let $\widetilde{\mathfrak{A}}$ and the $\mathfrak{A}_i$'s  be as above.
 Suppose that the  class
  $\mathfrak{A}_0 \in   Cl_k (\mathcal{O}_k) $ breaks up into finitely many
  smaller ray classes modulo $\mathfrak m$, say,
  $\mathfrak{B}_1 , \cdots, \mathfrak{B}_h \,$, where
  $ h = |Cl_k (\mathfrak m)|\,/\,|Cl_k (\mathcal{O}_k)|\,$.
Then we have
\begin{eqnarray*}
 \qquad \zeta^{\,*}( \,s \, , \widetilde{\mathfrak{A}} \,,
  \mathcal{O}_{\! A,{\mathfrak m}})
  \; = \; \sum_{j\,=\,1}^{h} \,\zeta_{\,k,\mathfrak m} ( \,s \,,\mathfrak{B}_j )\,
    \zeta_{\,k,\mathfrak m} ( \,s \,,  \mathfrak{B}_j \,\mathfrak{A}_1 )\,
    \cdots \,\zeta_{\,k,\mathfrak m} ( \,s \,, \mathfrak{B}_j \,\mathfrak{A}_n )\,.
 \end{eqnarray*}

\end{lemma}

\begin{proof}
 Let  $\tilde{\mathfrak{a}} \in \widetilde{\mathfrak{A}} $
  be an integral $\mathcal{O}_{\! A,{\mathfrak m}}$-ideal prime to
  ${\mathfrak m }_0 \mathcal{O}_{\!A}$.
 Put
  $\tilde{\mathfrak a} \,\mathcal{O}_{A} = ({\mathfrak a}_0 , \cdots ,  {\mathfrak a}_n ) $,
  where  the $ {\mathfrak a}_i $'s are integral ideals of $\mathcal{O}_k$
  prime to ${\mathfrak m}_0\,$.
 Then $ {\bf N} \tilde{\mathfrak{a}}
   = {\bf N}{\mathfrak a}_0  \cdots   {\bf N}{\mathfrak a}_n \,$.
 Moreover,  we have ${\mathfrak a}_0  \in \mathfrak{B}_j $
  for a unique $j\,$, $ 1 \leq j \leq h\,$, and
  ${\mathfrak a}_i \in \mathfrak{B}_j \,\mathfrak{A}_i \,$ for $1 \leq i\leq n \,$.
 Conversely, given $1 \leq j \leq h \,$, and given  integral
  $\mathcal{O}_k$-ideals ${\mathfrak a}_0 , \cdots ,  {\mathfrak a}_n  $
  prime to ${\mathfrak m}_0\,$ such that ${\mathfrak a}_0 \in \mathfrak{B}_j\,$,
  ${\mathfrak a}_i \in \mathfrak{B}_j \,\mathfrak{A}_i \,$ for $1 \leq i \leq n \,$,
  there exists a unique integral $\mathcal{O}_{\! A,{\mathfrak m}}$-ideal
  $\tilde{\mathfrak{a}} \in \widetilde{\mathfrak{A}} $ prime to
  ${\mathfrak m }_0 \mathcal{O}_{\!A}$
  under the $\mathcal{O}_{\!A}$-ideal
  $(\,{\mathfrak a}_0 , \cdots ,  {\mathfrak a}_n ) \,$.
  Hence  we have
\begin{eqnarray*}
  \zeta^{\,*}(\, s \, , \widetilde{\mathfrak{A}} \,,
  \mathcal{O}_{\! A,{\mathfrak m}})
   \; = \;\sum_{j\,=1}^{h}\,
    \sum_{
   \substack{  {\mathfrak a }_0 \in \mathfrak{B}_j   \\
   ( {\mathfrak a }_0 , {\mathfrak m }_0 ) =  1 } }
  { {\bf N} {\mathfrak a }_0^{ - s} }
  \sum_{
   \substack{  {\mathfrak a }_1 \in \mathfrak{B}_j \mathfrak{A}_1  \\
   ( {\mathfrak a }_1 , {\mathfrak m }_0 ) =  1 } }
  { {\bf N} {\mathfrak a }_1^{ - s} } \, \cdots
  \sum_{
   \substack{  {\mathfrak a }_n \in \mathfrak{B}_j \mathfrak{A}_n  \\
   ( {\mathfrak a }_n , {\mathfrak m }_0 ) =  1 } }
  { {\bf N} {\mathfrak a }_n^{ - s} } \, .
 \end{eqnarray*}

\end{proof}

 Let $\widetilde{\chi}$ be a character of $Cl_{A,\mathfrak m}$.
 Through the isomorphism $\Psi$, we may identify $\widetilde{\chi}$
  as a character of  $Cl_k (\mathcal{O}_k) \times Cl_k (\mathfrak m)^{n}$.
 Thus
 $\widetilde{\chi} = \chi_0^{\, } \times \chi_1^{\, } \times \cdots \times \chi_n \,$,
  where $\chi_0^{\, }$ is a character of $Cl_k (\mathcal{O}_k)\,$
  and $\chi_i$ ($ 1 \leq i \leq n $) are characters of $Cl_k (\mathfrak m)$.
\begin{theorem}
 With notation as above, we have
 \begin{eqnarray}
 \qquad \quad  L^{*}( \,s \, , \widetilde{\chi} \,,
  \mathcal{O}_{\! A,{\mathfrak m}} )
   =
    L_{k,\mathfrak m}( s \, , \chi_1 ) \,\cdots \,
       L_{k,\mathfrak m}( s \, , \chi_n )\,\, L_{k,\mathfrak m}
  ( \,s \, , \chi_0^{\, } \, \chi_1^{-1} \cdots \chi_n^{-1} \,)\,,
 \end{eqnarray}
where in the last $L$-function we identify $\chi_0^{\, }$ with its
induced character at
 $Cl_k (\mathfrak m)$ through the surjective homomorphism
  $ Cl_k (\mathfrak m) \rightarrow Cl_k (\mathcal{O}_k) \, $.
\end{theorem}

\begin{proof}
Lemma A$.3$ implies that
\begin{eqnarray*}
  \; L^{*}( \,s \, , \widetilde{\chi} \,,
  \mathcal{O}_{\! A,{\mathfrak m}} )
   & = & \!\!\!\! \sum_{\mathfrak{A}_0 \in   Cl_k (\mathcal{O}_k) }
    \sum_{\mathfrak{A}_1 \in   Cl_k (\mathfrak m) } \cdots
   \!\! \sum_{\mathfrak{A}_n \in   Cl_k (\mathfrak m) }
     \chi_0(\,\mathfrak{A}_0) \,
     \chi_1(\,\mathfrak{A}_1) \cdots
     \chi_n(\,\mathfrak{A}_n) \\
   &  & \quad \sum_{  \substack{ \mathfrak{B} \in  Cl_k (\mathfrak m) \\
                        \mathfrak{B} \subseteq \mathfrak{A}_0 } }
    \zeta_{\,k,\mathfrak m} ( \,s \,,\mathfrak{B} )\,
    \zeta_{\,k,\mathfrak m} ( \,s \,,  \mathfrak{B} \,\mathfrak{A}_1 )
    \cdots \zeta_{\,k,\mathfrak m} ( \,s \,,
    \mathfrak{B}\,\mathfrak{A}_n
    )\,.
\end{eqnarray*}
Regarding $\chi_0$ as a character defined at $Cl_k (\mathfrak m)\,$,
we may write the above  as
\begin{eqnarray*}
 \quad \sum_{   \mathfrak{B} \in  Cl_k (\mathfrak m) }
    \chi_0 ^{\,}\,
    \chi_1^{-1} \cdots
    \chi_n^{-1}(\,\mathfrak{B}\,)
    \,\zeta_{\,k,\mathfrak m} ( \,s \,,\mathfrak{B} \,)
    \sum_{\mathfrak{A}_1 \in   Cl_k (\mathfrak m) }
    \chi_1(\,\mathfrak{B}\,\mathfrak{A}_1) \,
    \zeta_{\,k,\mathfrak m} ( \,s \,,  \mathfrak{B} \,\mathfrak{A}_1 )
    \\
    \,\cdots
    \sum_{\mathfrak{A}_n \in   Cl_k (\mathfrak m) }
    \chi_n(\,\mathfrak{B}\,\mathfrak{A}_n) \,
    \zeta_{\,k,\mathfrak m} ( \,s \,,  \mathfrak{B} \,\mathfrak{A}_n )
    \,.
 \end{eqnarray*}
This is the right hand side of $($A$.1)$.
\end{proof}

\begin{corollary}
 Let $\chi$ be a character of $Cl_k (\mathfrak m)$ and
  let $\widetilde{\chi}$ be the character
  of $Cl_{A,\mathfrak m}$ corresponding to
  $ 1 \times \chi \times \cdots \times \chi $ under $\Psi$. Then
  we have
\begin{eqnarray*}
L_{k,\mathfrak m}
  ( \,s \, , \chi\,)^n
    \, L_{k,\mathfrak m}( \,s \, , \chi^{- n} ) =
    L^{*}( \,s \, , \widetilde{\chi} \,,
  \mathcal{O}_{\! A,{\mathfrak m}} )\,.
 \end{eqnarray*}
 In particular, if we choose $n$ to be the order of  $\,\chi\,$, then
 we get
\begin{eqnarray*}
L_{k,\mathfrak m}
  ( \,s \, , \chi\,)^n
    \, L_{k,\mathfrak m}( \,s \, , 1 ) =
    L^{*}( \,s \, , \widetilde{\chi} \,,
  \mathcal{O}_{\! A,{\mathfrak m}} )\,.
 \end{eqnarray*}
 \end{corollary}
 \noindent Theorem A$.4$  allows us  to compute
  $ L^{*}( \,s \, , \widetilde{\chi} \,, \mathcal{O}_{\! A,{\mathfrak m}})$
 from the $L_{k,\mathfrak m}( \,s \, , \chi )\,$'s and vice versa.
 In particular, let
  $A ={\mathbb Q}^2$, $m \in {\mathbb Z}^{+}$ and
  $\mathcal{O}_{\! A,{m}} = {\mathbb Z} {\,\bf 1} + m \,\mathcal{O}_{\! A}\,$.
 Then
  each  character $\widetilde{\chi}$ of $Cl_{A,m}$
  corresponds to $  1 \times \chi \,$
  for a unique character  $\chi$ of $Cl_{\mathbb Q} ( m)$ so that
  $\widetilde{\chi}$ is primitive if and only if  $\chi$ is primitive.
 Moreover, we have
\begin{eqnarray*}
 L^{*}( \,s \, , \widetilde{\chi}\,, \mathcal{O}_{\! A,{m}} )
  =  L_{{\mathbb Q}, m}( \,s \, , \chi\,)\,
   L_{{\mathbb Q}, m}( \,s \, , \chi^{- 1})\,.
 \end{eqnarray*}
 This result is used in Section $4$.

\subsection*{The relations between the $L$-series and the Truncated
 $L$-series of $\mathcal{O}_{\! A,{\mathfrak m}}\,$ }
 For the rest of the section we let
   $\,{\mathfrak m }' =  {\mathfrak m }_0 '\,{\mathfrak m }_\infty$
 denote a modulus in $k$ whose finite part
  ${\mathfrak m }_0 '$ is an integral ideal
  of $\mathcal{O}_k$ dividing ${\mathfrak m }_0\,$.
  We define
   $\mathcal{O}_{\! A,{\mathfrak m }'} $
   and  its narrow class group  $ Cl_{A,{\mathfrak m }'}$
   modulo ${\mathfrak m}_\infty$ as before.
 For
  such a modulus ${\mathfrak m }'$, we let
   $\varphi_{A,{\mathfrak m }'\!,{\mathfrak m }} :
     Cl_{A,{\mathfrak m }} \rightarrow  Cl_{A,{\mathfrak m }'}$
 denote the  homomorphism sending
 the narrow ideal class of $  \tilde{\mathfrak a}  $ modulo  ${\mathfrak m}_\infty$
 to the narrow ideal class of
    $ \tilde{\mathfrak a} \,\mathcal{O}_{\! A,{\mathfrak m }'}$
    modulo  ${\mathfrak m}_\infty\,$,\,
 for every $\tilde{\mathfrak a}\in I(\,\mathcal{O}_{\! A,{\mathfrak m }})\,$.
  Let $ U_{\! A,{\mathfrak m }'}$ denote  the group
  formed by units of $\mathcal{O}_{\! A,{\mathfrak m }'}$
  which are equivariant modulo ${\mathfrak m}_\infty\,$.
 It is easy to see that  $\varphi_{A,{\mathfrak m}'\!,{\mathfrak m }}$
  is surjective and induces the  exact sequence
\begin{eqnarray*}
 \quad \; 1 \,\longrightarrow \, U_{\! A,{\mathfrak m }} \,\longrightarrow\,
 U_{\! A,{\mathfrak m }'} \,\longrightarrow\,
 U(\,\mathcal{O}_{\! A,{\mathfrak m }'},\mathcal{O}_{\! A,{\mathfrak m }}\,)
  \,\longrightarrow\,
 Cl_{A,{\mathfrak m }}\, \longrightarrow \,Cl_{A,{\mathfrak m }'} \,\longrightarrow \, 1
 \,.
 \end{eqnarray*}
 Thus we have
 \begin{eqnarray}
 |\,U(\,\mathcal{O}_{\! A,{\mathfrak m }'},\mathcal{O}_{\! A,{\mathfrak m }}\,) \,| =
\,[\,U_{\! A,{\mathfrak m }'}\!:U_{\! A,{\mathfrak m }} ]
 \,| \,Cl_{A,{\mathfrak m }}\,|\, / \,| \,Cl_{A,{\mathfrak m }'}\,|\,.
 \end{eqnarray}

\begin{lemma}
 For each   class $ \widetilde{\mathfrak{A}} $ of $Cl_{A,\mathfrak m}$,
  we have
\begin{eqnarray*}
 \quad \zeta( \,s \, , \widetilde{\mathfrak{A}} \,,
  \mathcal{O}_{\! A,{\mathfrak m }}\, ) = \!\!\!\!
 \sum_{\mathfrak c \, \mathfrak m_0 '={\mathfrak m_0 }}
 \!\!
 \frac{ \,
  [\,\,U_{\! A,{\mathfrak m }'} :
   U_{\! A,{\mathfrak m }}\, ] \,
  }{\, {\bf N}{\mathfrak c}^{\, (n+1) s } }\,\,
  \zeta^{\,*}\!( \,s \, ,
  [ \,{\mathfrak c}^{-1}\mathcal{O}_{\! A,{\mathfrak m }'}  ]\,
  \varphi_{A,{\mathfrak m }'\!,{\mathfrak m }}
  (\, \widetilde{\mathfrak{A}}\, ) ,
  \mathcal{O}_{\! A,{\mathfrak m }'} )\,,
 \end{eqnarray*}
 where $ {\mathfrak m }' $ denotes
   the modulus $ {\mathfrak m }_0 '\,{\mathfrak m }_\infty$
   and the sum runs over the pairs of integral $\mathcal{O}_k$-ideals
   $\mathfrak c$ and $\mathfrak m_0 '$ with
  $ \mathfrak c \,\mathfrak m_0 '={\mathfrak m_0 }\,$.
\end{lemma}

\begin{proof}
 Let $ \tilde{\mathfrak a} $ be an integral invertible
  $\mathcal{O}_{\! A,{\mathfrak m }}$-ideal in the  class
  $\widetilde{{\mathfrak A }}$.
 Let $ {\mathfrak c} $ and $ {\mathfrak m }_0'$ be the integral $\mathcal{O}_k$-ideals
   given by
   $ {\mathfrak c}\,{\bf 1} =
   (\,\tilde{\mathfrak a} + {\mathfrak m }_0 \mathcal{O}_{A} \,)
    \,\cap \,\mathcal{O}_k \,{\bf 1} \,$
   and
   ${\mathfrak m }_0' = {\mathfrak m}_0\, /\, {\mathfrak c} \,$.
 Put ${\mathfrak m }' = {\mathfrak m }_\infty \, {\mathfrak m }_0'\,$.
 Then by   modular law (see Lemma $5.2$), we have
 $ \tilde{\mathfrak a} + {\mathfrak m }_0 \mathcal{O}_{A}
   =  {\mathfrak c} \,\mathcal{O}_{\! A,{\mathfrak m }'} \,$.
  Thus
   $ \tilde{{\mathfrak b }}={\mathfrak c}^{-1}\tilde{\mathfrak a}\,
      \mathcal{O}_{\! A,{\mathfrak m }'} $
  is an integral $ \mathcal{O}_{\! A,{\mathfrak m }'}$-ideal
  prime to
  $ {\mathfrak m }_0' \,\mathcal{O}_{\! A}$ lying in the  class
  $ [ \,{\mathfrak c}^{-1}\mathcal{O}_{\! A,{\mathfrak m }'}  ]\,
   \varphi_{A,{\mathfrak m }'\!,{\mathfrak m }}
    (\, \widetilde{\mathfrak{A}}\, )$.
  Moreover, we have ${\bf N}\tilde{{\mathfrak a }}
   ={\bf N}{\mathfrak c }^{\,n+1}\, {\bf N}\tilde{{\mathfrak b
   }}\,$. \\
 Conversely, given   integral $\mathcal{O}_k$-ideals
  $ {\mathfrak c} $ and $ {\mathfrak m }_0'$ with
  ${\mathfrak c} \, {\mathfrak m }_0' = {\mathfrak m }_0 \,$, and
  given   any integral $ \mathcal{O}_{\! A,{\mathfrak m }'}$-ideal
   $ \tilde{{\mathfrak b }}$
  prime to the conductor of $ \mathcal{O}_{\! A,{\mathfrak m }'}$
   (\,here
   ${\mathfrak m }' =  {\mathfrak m }_0'\,{\mathfrak m }_\infty $)\,,
    all the invertible $ \mathcal{O}_{\! A,{\mathfrak m}}$-ideals
  $\tilde{\mathfrak a }$ under
  $ {\mathfrak c }\,\tilde{\mathfrak b }$ are  integral ideals of
  $ \mathcal{O}_{\! A,{\mathfrak m}}$ and  satisfy
  \begin{eqnarray*}
   (\,\tilde{\mathfrak a} + {\mathfrak m }_0 \mathcal{O}_{A} \,)
  \,\cap \,\mathcal{O}_k \,{\bf 1} =   {\mathfrak c}\,{\bf 1} \,
  \end{eqnarray*}
  (see the proof of Lemma $5.3$).
 Suppose further that $\tilde{{\mathfrak b }}$ lies in
   $[ \,{\mathfrak c}^{-1}\mathcal{O}_{\! A,{\mathfrak m }'}  ]\,
   \varphi_{A,{\mathfrak m }'\!,{\mathfrak m }}
    (\, \widetilde{\mathfrak{A}}\, )$. This means that there exists
    some
 $\mathcal{O}_{\! A,{\mathfrak m }}$-invertible
 $\tilde{{\mathfrak a }} \in \widetilde{\mathfrak{A}} $ such that
  $ \tilde{{\mathfrak a }}\,\mathcal{O}_{\! A,{\mathfrak m }'}=
  {\mathfrak c }\,\tilde{\mathfrak b }\,$.
 Then under ${\mathfrak c }\,\tilde{\mathfrak b }\,$ there are
   precisely
   $[\,\,U_{\! A,{\mathfrak m }'} :
   U_{\! A,{\mathfrak m }}\, ]$
   invertible $\mathcal{O}_{\! A,{\mathfrak m }}$-ideals
  lying in the class $\widetilde{\mathfrak{A}}\,$. They are given by
  $\varepsilon \,\tilde{\mathfrak a} \,$, where $\varepsilon $ goes over a
  complete set of representatives of $ U_{\! A,{\mathfrak m }'}$
  modulo $U_{\! A,{\mathfrak m }}$.
The Lemma  follows by grouping the terms in
 $ \zeta( \,s \, , \widetilde{\mathfrak{A}} \,,
  \mathcal{O}_{\! A,{\mathfrak m }}\, )$ according to
  $ (\,\tilde{\mathfrak a} + {\mathfrak m }_0 \,\mathcal{O}_{A} \,)
  \,\cap \,\mathcal{O}_k \,{\bf 1} \,$.

\end{proof}

 In the following we let $\widetilde{\chi} $
  be a character of $Cl_{A,\mathfrak m}$.
 Through the isomorphism $\Psi$,
  we identify  $\widetilde{\chi} $ as a character of
   $ Cl_k (\,\mathcal{O}_k) \times Cl_k (\mathfrak m)^{n}$
  and write
   $\widetilde{\chi} = \chi_0^{\, } \times \chi_1^{\, } \times \cdots \times \chi_n $,
  where $\chi_0^{\, }$ is a character of $Cl_k (\,\mathcal{O}_k)\,$
  and $\chi_i$ ($ 1 \leq i \leq n $) are characters of $Cl_k (\mathfrak m)$.
 Let $ {\mathfrak f}_i $ denote the finite part of the conductor of $\chi_i\,$, $1\leq i \leq n\,$.
 Let $\mathfrak f $ denote the least common multiple ideal of
  $ \,\mathfrak f_1,\cdots,\mathfrak f_n \,$.
 It is easy to see that $\widetilde{\chi} $ is defined at $Cl_{A,\mathfrak m '}$
  with $ {\mathfrak m }' = {\mathfrak m }_0 '\, {\mathfrak m }_\infty $
  if and only if  $\,\mathfrak f \,|\,\mathfrak m_0 '$ and
    $\mathfrak m_0 ' |\,{\mathfrak m_0 }$.
 As before, we use the same symbol $\widetilde{\chi} $ to denote
  the character of $Cl_{A,\mathfrak m '}$ which induces $\widetilde{\chi}\,$.

\begin{theorem}
 With notation as above, we have
 \begin{eqnarray*}
 \quad
  L( \,s \, , \widetilde{\chi} \,,\mathcal{O}_{\! A,{\mathfrak m}} )
   =
   \sum_{\mathfrak f \,|\,\mathfrak m_0 ' |\,{\mathfrak m_0 }}
    \chi_0^{\,} (\,{\mathfrak c }) \,
 \frac{ \,
  |\,U(\,\mathcal{O}_{\! A,{\mathfrak m }'},
        \mathcal{O}_{\! A,{\mathfrak m }}\,)
  \,|\,}
  {\,\, {\bf N}{\mathfrak c }^{\, (n+1) s } } \,
 L^{*}( \,s \, ,
   \widetilde{\chi} \,,\mathcal{O}_{\! A,{\mathfrak m}'} )\,,
  \end{eqnarray*}
  where $ \,{\mathfrak m }' =   {\mathfrak m }_0 '\,{\mathfrak m}_\infty\,$,
    ${\mathfrak c }= {\mathfrak m_0^{\,} }/\mathfrak m_0 '$
   and the sum extends over  integral $\mathcal{O}_k$-ideals
   $\mathfrak m_0'$ such that
    $\,\mathfrak f \,|\, \mathfrak m_0 '$ and
    $\mathfrak m_0 ' \,|\,{\mathfrak m_0 }\,$.
 In particular, if $\,\mathfrak f = \mathfrak m_0\,$, then we have
\begin{eqnarray*}
 L( \,s \, , \widetilde{\chi} \,,\mathcal{O}_{\! A,{\mathfrak m}} ) =
  L^{*}( \,s \, , \widetilde{\chi} \,,\mathcal{O}_{\! A,{\mathfrak m}} )\,.
 \end{eqnarray*}
\end{theorem}

\begin{proof}
 By Lemma A$.6$, we can write
  $ L( \,s \, , \widetilde{\chi} \,,\mathcal{O}_{\! A,{\mathfrak m}} )$
   as
 \begin{eqnarray*}
 & & \sum_{\widetilde{\mathfrak{A}} \in Cl_{A,{\mathfrak m }} }
  \widetilde{\chi}(\,\widetilde{\mathfrak{A}}\,) \,
  \sum_{\mathfrak c \,\mathfrak m_0 '={\mathfrak m_0 }}
 \!\!
 \frac{ \,
  [\,U_{\! A,{\mathfrak m }'} :
   U_{\! A,{\mathfrak m }}\, ] \,
  }{\, {\bf N}{\mathfrak c}^{\, (n+1) s } }\,\,
  \zeta^{\,*}\!( \,s \, ,
  [ \,{\mathfrak c}^{-1}\mathcal{O}_{\! A,{\mathfrak m }'}  ]\,
  \varphi_{A,{\mathfrak m }'\!,{\mathfrak m }}
  (\, \widetilde{\mathfrak{A}}\, ) ,
  \mathcal{O}_{\! A,{\mathfrak m }'} )\, \\ [4pt]
  &=&
  \sum_{\mathfrak c \,\mathfrak m_0 '={\mathfrak m_0 }}
  \!\!
 \frac{ \,
  [\,U_{\! A,{\mathfrak m }'} :
   U_{\! A,{\mathfrak m }}\, ] \,
  }{\, {\bf N}{\mathfrak c}^{\, (n+1) s } }\,\,
  \sum_{\widetilde{\mathfrak{A}}' \in Cl_{A,{\mathfrak m }'} }
   \!\!
    \sum_{ \substack{
     \widetilde{\mathfrak{A}}\, \in \,Cl_{A,{\mathfrak m }} \\
       \varphi_{A,{\mathfrak m }'\!,{\mathfrak m }}
  ( \widetilde{\mathfrak{A}} )
  = [ {\mathfrak c}\,\mathcal{O}_{\! A,{\mathfrak m }'}  ]\,
    \widetilde{\mathfrak{A}}' } }
     \widetilde{\chi}(\,\widetilde{\mathfrak{A}}\,)
   \,
  \zeta^{\,*}\!( \,s \, ,
    \widetilde{\mathfrak{A}}' ,
  \mathcal{O}_{\! A,{\mathfrak m }'} )\,.
\end{eqnarray*}
 For each $\widetilde{\mathfrak{A}}'\in Cl_{A,{\mathfrak m }'}$,
 there  is a $ \,\widetilde{\mathfrak{A}}^* \in Cl_{A,{\mathfrak m }}\,$
 such that
 $ \varphi_{A,{\mathfrak m }'\!,{\mathfrak m }}\,( \,\widetilde{\mathfrak{A}}^* )
  \, = \, [ \,{\mathfrak c}\,\mathcal{O}_{\! A,{\mathfrak m }'}  ]\,
   \widetilde{\mathfrak{A}}'$.
 So we have
\begin{eqnarray*}
  \sum_{ \substack{
     \widetilde{\mathfrak{A}} \,\in \, Cl_{A,{\mathfrak m }} \\
       \varphi_{A,{\mathfrak m }'\!,{\mathfrak m }}
  ( \widetilde{\mathfrak{A}} )
  \,=\, [ {\mathfrak c}\,\mathcal{O}_{\! A,{\mathfrak m }'}  ]\,
    \widetilde{\mathfrak{A}}' } } \!\!\!\!\!\! \!\!\!\!\!\!  \!\!
     \widetilde{\chi}(\,\widetilde{\mathfrak{A}}\,)
   & = &  \!\! \widetilde{\chi}(\,\widetilde{\mathfrak{A}}^*) \!\!
   \sum_{\widetilde{\mathfrak{B}} \,\in \,
   \mbox{\scriptsize{Ker}}\,\varphi_{A,{\mathfrak m }'\!,{\mathfrak m
   }}} \!\!\!\!\!\!
     \widetilde{\chi}(\,\widetilde{\mathfrak{B}}\,)\\
   & = & \!\! \!\!
 \left\{
   \begin{array}{ll}
     \widetilde{\chi}(\,\widetilde{\mathfrak{A}}^*)
     \,|\,\mbox{Ker}\,
     \varphi_{A,{\mathfrak m }'\!,{\mathfrak m
   }}\,|\,,
   &  \mbox{if $\widetilde{\chi}$ is trivial on
   $\mbox{Ker}\,\varphi_{A,{\mathfrak m }'\!,{\mathfrak m
   }}$};  \\[4pt]
      \qquad \qquad 0\,, & \mbox{ otherwise\,. }
   \end{array}
\right.
\end{eqnarray*}
 The  sum above vanishes unless $\widetilde{\chi}$ is defined at
 $Cl_{A,{\mathfrak m }'}$.
 Thus, using  (A$.2$), we get
\begin{eqnarray*}
  L( \,s \, , \widetilde{\chi} \,,\mathcal{O}_{\! A,{\mathfrak m}} ) =
  \!\! \!\!
 \sum_{ \substack{ \mathfrak c \,\mathfrak m_0 '={\mathfrak m_0 } \\
                {\mathfrak f }\,|\,\mathfrak m_0 '   }    }
  \!\! \!\!
  \widetilde{\chi}(\,{\mathfrak c}\mathcal{O}_{\! A,{\mathfrak m }'}) \,
 \frac{ \,
  |\,U(\mathcal{O}_{\! A,{\mathfrak m }'},\mathcal{O}_{\! A,{\mathfrak m }}) \,|
  }{\, {\bf N}{\mathfrak c}^{\, (n+1) s } } \!\!\!\!\!
  \sum_{\widetilde{\mathfrak{A}}' \in Cl_{A,{\mathfrak m }'} }\!\!\!\!\!\!
   \widetilde{\chi}(\,\widetilde{\mathfrak{A}}')\,
  \zeta^{\,*}\!( \,s \, ,
    \widetilde{\mathfrak{A}}'\! ,
  \mathcal{O}_{\! A,{\mathfrak m }'} ),
\end{eqnarray*}
the Theorem follows.
\end{proof}

\begin{remark}
 It is easy to prove that
 \begin{eqnarray*}
 |\,U(\,\mathcal{O}_{\! A,{\mathfrak m }'},
  \mathcal{O}_{\! A,{\mathfrak m }}\,) \,| =
  \Big(\,
  {\bf N} \big( \, {\mathfrak m}_0 \, / \, {\mathfrak m}_0' \,\big)
   \prod_{ \substack{ {\mathfrak p}\, |\, {\mathfrak m}_0  \\
                      {\mathfrak p}\, \nmid \, {\mathfrak m}_0 '} }
   \big( \,1 - \frac{1}{\, {\bf N}{\mathfrak p} }  \,\big) \, \Big)^n \,.
\end{eqnarray*}
\end{remark}

 Let $\mu$ denote the M${\ddot{\mbox{o}}}$bius function on integral
 ideals of ${{\mathcal{O}}_k}\,$, that is, $ \mu  $ is a function
 satisfying
  $\, \mu ( \,{{\mathcal{O}}_k} ) = 1 \,, \; \mu (\, {\mathfrak p} ) = -1\,$
  for each prime ideal $ {\mathfrak p} $ of ${{\mathcal{O}}_k}\,$,
  and $\, \mu (\, {\mathfrak p} ^{\,l} ) = 0 \; $  for all integers $ l > 1 \,$.
 Moreover,
  $  \mu ( {\mathfrak a} ) \,\mu ( {\mathfrak b } ) = \mu (\, {\mathfrak a}\, {\mathfrak b } ) \,$
  if  ${\mathfrak a} $ and $ {\mathfrak b } \,$ are coprime integral ideals.
 The M${\ddot{\mbox{o}}}$bius function $\mu$ satisfies the  relation
 \begin{eqnarray*}
 \sum_{ \mathfrak c  | \, \mathfrak c ' }
   \,\mu (\,\mathfrak c \,)
 \, =
 \left\{
   \begin{array}{ll}
    1\,, & \; \mbox{if $\mathfrak c ' =\mathcal{O}_k $}\,;  \\[4pt]
    0\,, & \mbox{ otherwise\,. }
       \end{array}
\right.
\end{eqnarray*}
 for all integral ideals  $\mathfrak c'$ of $\mathcal{O}_k\,$.
 From this relation  we deduce immediately  the following inversion formula.
\begin{corollary}
 With notation as above, we have
 \begin{eqnarray*}
  L^{*}( \,s \, , \widetilde{\chi} \,,\mathcal{O}_{\! A,{\mathfrak m}} )
   =
   \sum_{\mathfrak f \,|\,{\mathfrak m }_0 ' |\,{\mathfrak m_0 }}
   \mu (\,{\mathfrak c }) \,
    \chi_0^{\,} (\,{\mathfrak c }) \,
 \frac{ \,
  |\,U(\,\mathcal{O}_{\! A,{\mathfrak m }'},
        \mathcal{O}_{\! A,{\mathfrak m }}\,)
  \,|\,}
  {\, {\bf N}{\mathfrak c }^{\, (n+1) s } } \,
 L( \,s \, ,
   \widetilde{\chi} \,,\mathcal{O}_{\! A,{\mathfrak m}'} )\,,
  \end{eqnarray*}
 where in  the sum above
  $ \,{\mathfrak m }' =  {\mathfrak m }_0 '\,{\mathfrak m }_\infty $
  and
  $ {\mathfrak c } = {\mathfrak m_0 }/\mathfrak m_0 ' \,$.
\end{corollary}

\bibliographystyle{amsplain}

\section*{Acknowledgment}
The author would like to thank Frank Thorne for helpful comments
related to the topics of this paper.

\end{document}